\documentclass[10pt, reqno]{amsart}

\usepackage{amsmath, bm}
\usepackage{amssymb}
\usepackage{array}
\usepackage{caption} 
\usepackage{enumitem}
\usepackage{hyperref}
\usepackage[capitalise, noabbrev]{cleveref}
\usepackage{microtype}
\usepackage{mathrsfs}

\usepackage{tikz}
\usetikzlibrary{graphs,arrows,decorations.pathmorphing,backgrounds,positioning,fit,chains,matrix,scopes,decorations.markings}
\usepackage{braids}

\newtheorem{theorem}{Theorem}[section]

\newtheorem{lemma}[theorem]{Lemma}
\newtheorem{proposition}[theorem]{Proposition}
\newtheorem{corollary}[theorem]{Corollary}

\theoremstyle{definition}
\newtheorem{case}{Case}
\theoremstyle{remark}
\newtheorem{remark}[theorem]{Remark}

\numberwithin{equation}{section}

\newcommand{\ad}{\text{ad}} 
\newcommand{\Ad}{\text{Ad}} 
\newcommand{\Aez}{\mathcal{A}_{\be,\bz}}

\newcommand{\Bcs}{B_{\mathbf{c},\mathbf{s}}} 
\newcommand{\bc}{{\mathbf c}} 
\newcommand{\be}{{\boldsymbol{\epsilon}}} 
\newcommand{\bet}{\boldsymbol{\eta}}
\newcommand{\bs}{{\mathbf s}} 
\newcommand{\bz}{\boldsymbol{\zeta}}
\newcommand{\Bcc}{B_{\mathbf{c}}} 
\newcommand{\Beps}{B_{\be}}

\newcommand{\C}{\mathbb{C}} 
\newcommand{\ct}{\mathcal{T}} 
\newcommand{\cs}{\epsilon} 

\newcommand{\field}{\mathbb{K}} 

\newcommand{\G}{\Gamma}

\newcommand{\K}[1]{L_{#1}} 
\newcommand{\Kw}[1]{K_{\varpi_{#1}'}} 

\newcommand{\lie}[1]{\mathfrak{#1}} 

\newcommand{\N}{\mathbb{N}} 

\renewcommand{\O}{\Omega^- \!\!- \Omega^+}

\renewcommand{\t}{\tau} 

\newcommand{\Uqg}{U_q(\lie{g})} 
\newcommand{\Uow}{U_0^{\Theta}} 

\newcommand{\widet}{\widetilde} 

\newcommand{\Z}{\mathcal{Z}} 

\title{Braid group actions for quantum symmetric pairs of type AIII/AIV}
\author{Liam Dobson}
\address{Liam Dobson, School of Mathematics, Statistics and Physics, Newcastle University, Herschel Building, Newcastle upon Tyne NE1 7RU, UK}
\email{liam.dobson.maths@gmail.com}


\begin{document}

\begin{abstract}
   In the present paper we construct braid group actions on quantum symmetric pair coideal subalgebras of type AIII/AIV.
   This completes the proof of a conjecture by Kolb and Pellegrini in the case where the underlying Lie algebra is $\lie{sl}_n$.
   The braid group actions are defined on the generators of the coideal subalgebras and the defining relations and braid relations are verified by explicit calculations.
\end{abstract}	

\maketitle

\section{Introduction}
\subsection{Background}
Let $\lie{g}$ be a complex semisimple Lie algebra and $\Uqg$ the corresponding Drinfeld-Jimbo quantised enveloping algebra. In the theory of quantum groups, a crucial role is played by Lusztig's braid group action on $\Uqg$, see \cite{b-Lus94}. This braid group action provides an algebra automorphism $T_w$ of $U_q(\lie{g})$ for each element $w$ in the Weyl group $W$ of $\lie{g}$.

Let $\theta: \lie{g} \rightarrow \lie{g}$ be an involutive Lie algebra automorphism and let $\lie{k} = \{ x \in \lie{g} \mid \theta(x) = x \}$ denote the fixed Lie subalgebra. Recall from \cite{a-Araki62} that involutive automorphisms of $\lie{g}$ are parameterised up to conjugation by combinatorial data $(X, \t)$ attached to the Dynkin diagram of $\lie{g}$. Here $X \subset I$ where $I$ denotes an index set for the nodes of the Dynkin diagram of $\lie{g}$, and $\t: I \rightarrow I$ is a diagram automorphism. In a series of papers, G. Letzter constructed and investigated quantum group analogues of $\lie{k}$, see \cite{a-Letzter99a, MSRI-Letzter, a-Letzter03}. More precisely, she defined families of coideal subalgebras $\Bcs = \Bcs(X,\t) \subset \Uqg$  which are quantum group analogues of $U(\lie{k})$ depending on parameters $\bc$ and $\bs$. The algebras $\Bcs$ can be described explicitly in terms of generators and relations. We call $(\Uqg, \Bcs)$ a quantum symmetric pair and we refer to $\Bcs$ as a quantum symmetric pair coideal subalgebra.

There exists a braid group action on the fixed Lie subalgebra $\lie{k}$ by Lie algebra automorphisms. Let $W$ denote the Weyl group of $\lie{g}$ with corresponding braid group $Br(\lie{g})$, generated by elements $\{ \varsigma_i \mid i \in I\}$. We write $Br(W_X)$ to denote the subgroup of $Br(\lie{g})$ generated by $\{\varsigma_i \mid i \in X\}$; this is the braid group corresponding to the parabolic subgroup $W_X \subset W$. Further, associated to the pair $(X, \t)$ is a restricted root system $\Sigma$ with Weyl group $\widet{W}$ generated by elements $\widet{\sigma}_i$ parameterised by the $\t$-orbits in $I \setminus X$. The group $\widet{W}$ can be considered as a subgroup of $W$. Let $Br(\widet{W}) \subset Br(\lie{g})$ denote the corresponding braid group, generated by elements $\widet{\varsigma}_i$. Then the semidirect product $Br(W_X) \rtimes Br(\widet{W}) \subset Br(\lie{g})$ acts on $\lie{k}$ by Lie algebra automorphisms.

It was conjectured by Kolb and Pellegrini that there exists a quantum group analogue of this action on $\Bcs$ by algebra automorphisms \cite[Conjecture~1.2]{a-KP11}. This conjecture has been proved in type AII and in all cases where $X = \emptyset$ with the help of computer calculations \cite{a-KP11}. In the case $\lie{g} = \lie{sl}_n$, this leaves one substantial case in Araki's list \cite[p.32]{a-Araki62}, namely the type AIII/AIV with $X \neq \emptyset$. This case is shown in Figure \ref{Fig:AIII}.

\subsection{Results}
In the present paper we construct an action of $Br(W_X) \rtimes Br(\widet{W})$ on $\Bcs$ by algebra automorphisms in type AIII/AIV, hence completing the proof of Kolb and Pellegrini's conjecture for $\lie{g} = \lie{sl}_n$. 
    \begin{figure}[h!] 
        \centering
       \begin{tikzpicture}
            [white/.style={circle,draw=black,inner sep = 0mm, minimum size = 3mm},
		    black/.style={circle,draw=black,fill=black, inner sep= 0mm, minimum size = 3mm}]
		
		    \node[white] (first) [label = above:{\scriptsize $1$}] {};		
		    \node[white] (second) [right=of first] [label = above:{\scriptsize $2$}] {}
			    edge (first);
		    \node[white] (third) [right = 1.5cm of second] [label = above:{\scriptsize $r$}] {}
			    edge [dashed] (second);
			
		    \node[black] (fourth) [below right = 0.2cm and 0.3cm of third] {}
			    edge (third);
		    \node[black] (fifth) [below = of fourth] {}
			    edge [dashed] (fourth);
			
		    \node[white] (sixth) [below left = 0.2cm and 0.3cm of fifth] [label =                      below:{\scriptsize $n-r+1$}] {}
			    edge (fifth)
			    edge	 [latex'-latex' , shorten <=3pt, shorten >=3pt, bend left=30, densely dotted] node[auto,swap] {} (third);
		    \node[white] (seventh) [left =1.5cm of sixth] [label = below:{\scriptsize $n-1$}] {}
			    edge [dashed] (sixth)
			    edge	 [latex'-latex' , shorten <=3pt, shorten >=3pt, bend left=30, densely dotted] node[auto,swap] {} (second);
		    \node[white] (last) [left = of seventh] [label = below:{\scriptsize $n$}] {}
			    edge (seventh)
			    edge	 [latex'-latex' ,shorten <=3pt, shorten >=3pt, bend left=30, densely dotted] node[auto,swap] {} (first);
        \end{tikzpicture}
        \caption{The Satake diagram of type AIII/AIV} \label{Fig:AIII}
    \end{figure}
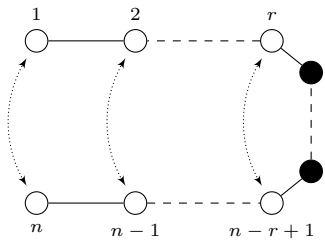
In this case $Br(W_X)$ is a classical braid group in $n-2r+1$ strands and $Br(\widet{W})$ is isomorphic to an annular braid group in $r$ strands. The subgroups $Br(W_X)$ and $Br(\widet{W})$ of $Br(\lie{g})$ commute and hence their semidirect product is just a direct product. Moreover, the parameters $\bs$ satisfy $\bs = (0, 0, \dotsc, 0)$ and hence we write $\Bcc = B_{\bc, \mathbf{0}}$. The following is the main result of this paper.

\medskip

\noindent \textbf{Theorem A.} (Theorem~3.8) \textit{Let $(X,\t)$ be a Satake diagram of type AIII/AIV with $X \neq \emptyset$. Then there exists an action of $Br(W_X) \times Br(\widet{W})$ on $\Bcc$ by algebra automorphisms. The action of $Br(\widet{W})$ on $\Bcc$ is given by algebra automorphisms $\ct_i$ defined by \cref{Eqn:cti,Eqn:ctr}.} 

\medskip

The action of $Br(W_X)$ on $\Bcc$ coincides with the Lusztig action, see \cite[Section~4.1]{a-BW16}. However, the action of $Br(\widet{W}) \subset Br(\lie{g})$ on $\Uqg$ does not restrict to an action on the coideal subalgebra $\Bcc$. Taking guidance from \cite{a-KP11}, the Lusztig automorphisms corresponding to elements of $\widet{W} \subset W$ are used in order to construct algebra automorphisms $\ct_i$ for $1 \leq i \leq r$. In particular, let $T_{\widet{\sigma_i}}$ for $1 \leq i \leq r$ be the Lusztig automorphism defined by \eqref{Eqn:Tw} and \eqref{Eqn:widet_sigma}. We then construct elements $\ct_i(B_j)$ for $j \in I \setminus X$ such that $\ct_i(B_j)$ and $T_{\widet{\sigma}_i}(B_j)$ have identical terms containing maximal powers of the generators $F_k$ of $\Uqg$ for $k \in I$, up to a factor.
For $1 \leq i \leq r-1$ the results of \cite{a-KP11} imply that $\ct_i$ defines an algebra automorphism of $\Bcc$. In this case, our definition of $\ct_i(B_j)$ coincides with the definition in \cite[Equation~4.11]{a-KP11} up to a factor.

The proof of Theorem A proceeds in three steps. First we verify that the formulae for $\ct_r$ given in \eqref{Eqn:ctr} define an algebra automorphism of $\Bcc$. Secondly, we show that the algebra automorphisms $\ct_i$ for $1 \leq i \leq r$ satisfy the braid relations for $Br(\widet{W})$. By \cite[Section~4.2]{a-KP11} we only need to show that $\ct_r\ct_{r-1}\ct_r\ct_{r-1} = \ct_{r-1}\ct_r\ct_{r-1}\ct_r$ and $\ct_i\ct_r = \ct_r\ct_i$ for $1 \leq i \leq r-2$ holds on each generator of $\Bcc$. This gives an action of $Br(\widet{W})$ on $\Bcc$ by algebra automorphisms. Finally, we show directly that the actions of $Br(W_X)$ and $Br(\widet{W})$ commute by case-by-case checks. 

We emphasise that all of the original results of the present paper are established without the use of computer calculations. 
The fact that the maps $\ct_i$ for $1 \leq i \leq r-1$ define algebra automorphisms of $\Bcc$ translates from \cite{a-KP11} where it was verified by computer calculations using the package \verb|QUAGROUP| of the computer algebra program \verb|GAP|. We do not reprove this fact. However, the calculations in the present paper suggest that one can prove Theorem \ref{Thm:cti} (\cite[Theorem~4.6]{a-KP11}) without the use of computer calculations.

In \cite{a-LW19a}, M. Lu and W. Wang developed a Hall algebra approach to the construction of quantum symmetric pairs with $X = \emptyset$ for $\lie{g}$ of type ADE (also excluding type $\text{A}_n$ for $n$ even if $\t \neq \text{id}$). In this setting they subsequently constructed Bernstein-Gelfand-Ponomarev type reflection functors in \cite{a-LW19b} which recover the corresponding braid group action in \cite{a-KP11}. At the end of \cite[Section~1.5]{a-LW19b}, they express great interest to develop this approach fully to cover general Satake diagrams with $X \neq \emptyset$. The braid group action for quantum symmetric pairs of type AIII/AIV with $X \neq \emptyset$ constructed in the present paper provides a crucial test case for any such generalisations. Formula \eqref{Eqn:ctr} indicates that the general setting will be substantially more complicated.

\subsection{Organisation}
In Section \ref{Sec:Prelims} we recall fixed Lie subalgebras of type AIII/AIV and their corresponding quantum symmetric pairs. We also recall the braid group action of $Br(W_X) \times Br(\widet{W})$ on the fixed Lie subalgebra in this case.

In \cref{Sect:Braid_Action} we present the main results of the present paper. We recall the action of $Br(W_X)$ on $\Bcc$ and define the algebra automorphisms $\ct_i$, giving an action of $Br(\widet{W})$ on $\Bcc$. We also show that the two actions commute. In \cref{Sec:Proof1,Sec:Proof2} we prove that $\ct_r$ is an algebra automorphism of $\Bcc$ and that the automorphisms $\ct_i$ satisfy the braid relations for $Br(\widet{W})$, respectively. This requires the use of many involved relations in $\Bcc$, which are given in \cref{Appendix}.

\subsection*{Acknowledgement} The author is grateful to Stefan Kolb for useful comments and advice.

\section{Preliminaries} \label{Sec:Prelims}
\subsection{Braid group actions on fixed Lie subalgebras of type AIII/AIV} \label{Sec:TypeAIII}
Let $\lie{g} = \lie{sl}_{n+1}(\C)$ for $n \in \N$ with Cartan subalgebra $\lie{h}$ consisting of traceless diagonal $(n+1) \times (n+1)$ matrices. 
Let $\Phi \subset \lie{h}^{\ast}$ be the corresponding root system. 
Choose a set $\Pi = \{\alpha_i \mid i \in I\}$ of simple roots where $I = \{1, 2, \dotsc, n\}$ denotes an index set for the nodes of the Dynkin diagram of $\lie{g}$. 
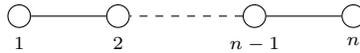
    \begin{center}
       \begin{tikzpicture}
            [white/.style={circle,draw=black,inner sep = 0mm, minimum size = 3mm},
		    black/.style={circle,draw=black,fill=black, inner sep= 0mm, minimum size = 3mm}]
		
		    \node[white] (first) [label = below:{\scriptsize $1$}] {};		
		    \node[white] (second) [right=of first] [label = below:{\scriptsize $2$}] {}
			    edge (first);
	        \node[white] (third) [right = 1.5cm of second] [label = below:{\scriptsize $n-1$}] {}
	            edge [dashed] (second);
	        \node[white] (last) [right = of third] [label = below:{\scriptsize $n$}] {}
	            edge (third);
        \end{tikzpicture}
    \end{center}
Let $Q = \mathbb{Z}\Pi$ denote the root lattice of $\lie{g}$. 
Let $\varpi_i \in \lie{h}^{\ast}$ denote the $i^{\text{th}}$ fundamental weight and let $P = \sum_{i \in I}\mathbb{Z}\varpi_i$ denote the weight lattice of $\lie{g}$. 
Write $W$ to denote the Weyl group of $\lie{g}$, generated by reflections $\sigma_i$ for $i \in I$. 
Fix a $W$-invariant scalar product $(-,-)$ on the real vector space spanned by $\Phi$ such that $(\alpha, \alpha) = 2$ for all roots $\alpha$. 
For $i, j \in I$ let
    \begin{equation} \label{Eqn:Cartan}
        a_{ij} =
            \begin{cases}
                2 & \mbox{if $i = j$,}\\
                -1 & \mbox{if $|i-j| = 1$,}\\
                0 & \mbox{otherwise}
            \end{cases}
    \end{equation}
denote the entries of the Cartan matrix of $\lie{g}$. 
Let $Br(\lie{g})$ denote the classical braid group corresponding to $\lie{g}$. 
This is the group generated by elements $\{\varsigma_i \mid i \in I\}$ subject to relations
    \begin{align} 
        \varsigma_i\varsigma_j &=\varsigma_j\varsigma_i & &\mbox{if $a_{ij} = 0$,} \label{Eqn:Braid_Rel1}\\
        \varsigma_i\varsigma_j\varsigma_i &= \varsigma_j\varsigma_i\varsigma_j & &\mbox{if $a_{ij} = -1$.} \label{Eqn:Braid_Rel2}
    \end{align}
Let $\{e_i, f_i, h_i \mid i \in I\}$ denote a set of Chevalley generators of $\lie{g}$ and define
    \begin{equation} \label{Eqn:Ad}
        \Ad(\varsigma_i) = \exp(\ad(e_i)) \exp(\ad(-f_i)) \exp(\ad(e_i))
    \end{equation}
for $i \in I$ where $\exp$ denotes the exponential series and $\ad$ denotes the adjoint action. Then there exists a group homomorphism
\begin{equation} \label{Eqn:Steinberg_Action}
	\textup{\Ad}: Br(\lie{g}) \rightarrow \textup{Aut}(\lie{g})
\end{equation}
such that $\textup{\Ad}(\varsigma_i)$ is given by \cref{Eqn:Ad}, see \cite[Lemma~56]{b-Ste67}.

Let $\theta: \lie{g} \rightarrow \lie{g}$ be an involutive Lie algebra automorphism and let $\lie{k} = \{ x \in \lie{g} \mid \theta(x) = x\}$ denote the corresponding fixed Lie subalgebra. 
Recall from  \cite[Section~7]{a-Letzter03} and \cite[Section~2.4]{a-Kolb14} that involutive automorphisms of $\lie{g}$ are classified up to conjugation via Satake diagrams $(X,\t)$ where $X \subset I$ and $\t: I \rightarrow I$ is a diagram automorphism. 
Throughout this paper, we consider Satake diagrams of type AIII/AIV, as indicated by \cite[Table~1]{a-Araki62} and Figure \ref{Fig:AIII}. In particular we fix $r \in \N$ such that $1 \leq r \leq \lceil \frac{n}{2} \rceil - 1$ and let $X = \{ r+1, r+2, \dotsc, n-r\} \neq \emptyset$. The diagram automorphism $\t$ is given by 
    \begin{equation} \label{Eqn:diagaut}
        \t(i) = n-i+1
    \end{equation}
for each $i \in I$. 
This can be lifted to a Lie algebra automorphism, also denoted by $\t$. For any $J \subset I$, let $W_J$ denote the parabolic subgroup of $W$ generated by $\{ \sigma_i \mid i \in J \}$ and let $Br(W_J)$ denote the associated braid group, generated by $\{ \varsigma_i \mid i \in J \}$. We denote by $w_J$ and $m_J$ the longest element in $W_J$ and the corresponding element of $Br(W_J)$, respectively.

By \cite[Theorem~2.5]{a-Kolb14}, the involution $\theta$ is given by
\begin{equation} \label{Eqn:theta}
	\theta = \Ad(s) \circ \Ad(m_X) \circ \omega \circ \t
\end{equation}
where $\Ad(s):\lie{g} \rightarrow \lie{g}$ is a Lie algebra automorphism such that restriction of $\Ad(s)$ to any root space is given by multiplication by a scalar, see \cite[Section~5.1]{a-BK19} and $\omega:\lie{g} \rightarrow \lie{g}$ denotes the Chevalley involution.

Generally, the braid group action on $\lie{g}$ given by Equation \eqref{Eqn:Steinberg_Action} does not restrict to an action on $\lie{k}$.
We consider instead a suitable subgroup of $Br(\lie{g})$ that depends on $X \subset I$ and $\t: I \rightarrow I$. For any $1 \leq i \leq r$ let
    \begin{equation} \label{Eqn:widet_sigma}
        \widet{\sigma}_i = w_{ \{i, \t(i)\} \cup X }w_X^{-1}
            =   \begin{cases}
                    \sigma_i\sigma_{\t(i)} & \mbox{if $1 \leq i \leq r-1$,}\\
                    \sigma_r\sigma_{r+1} \dotsm \sigma_{n-r+1} \dotsm \sigma_{r+1}\sigma_r & \mbox{if $i=r$}
                \end{cases}
    \end{equation}
and denote by $\widet{W}$ the subgroup of $W$ generated by $\{ \widet{\sigma}_i \mid 1 \leq i \leq r \}$. The subgroup $\widet{W}$ can be interpreted as the Weyl group of the restricted root system $\Sigma$ of the symmetric Lie algebra $(\lie{g}, \theta)$, see  \cite[Section~2.2]{a-DK18}. Let $Br(\widet{W})$ denote the subgroup of $Br(\lie{g})$ generated by the elements
    \begin{equation} \label{Eqn:widet_braid}
        \widet{\varsigma}_i = m_{ \{i, \t(i)\} \cup X }m_X^{-1}
            =   \begin{cases}
                    \varsigma_i\varsigma_{\t(i)} & \mbox{if $1 \leq i \leq r-1$,}\\
                    \varsigma_r\varsigma_{r+1} \dotsm \varsigma_{n-r+1} \dotsm \varsigma_{r+1}\varsigma_r & \mbox{if $i=r$.}
                \end{cases}
    \end{equation}
The elements $\widet{\varsigma}_i$ satisfy the relations
    \begin{align*}
        \widet{\varsigma}_i\widet{\varsigma}_j &= \widet{\varsigma}_j\widet{\varsigma}_i & &\mbox{if $a_{ij} = 0$ and $1 \leq i,j \leq r$,}\\
        \widet{\varsigma}_i\widet{\varsigma}_j\widet{\varsigma}_i &= \widet{\varsigma}_j\widet{\varsigma}_i\widet{\varsigma}_j & &\mbox{if $a_{ij}=-1$ and $1 \leq i,j < r$,} \\
        \widet{\varsigma}_i\widet{\varsigma}_j\widet{\varsigma}_i\widet{\varsigma}_j &= \widet{\varsigma}_j\widet{\varsigma}_i\widet{\varsigma}_j\widet{\varsigma}_i & &\mbox{if $i = r, j = r-1$.}
    \end{align*}

Since $Br(W_X)$ and $B_r(\widet{W})$ commute, we consider the subgroup $Br(W_X) \times Br(\widet{W})$.
We state without proof the version of \cite[Lemma~2.1]{a-KP11} corresponding to the present case. In many cases, $\Ad(s)$ appearing in Equation \eqref{Eqn:theta} does not commute with $\Ad(b)$ for $b \in Br(W_X) \times Br(\widet{W})$. For this reason, conjugating the action $\Ad$ by a Lie algebra isomorphism $\psi_s:\lie{g} \rightarrow \lie{g}$ depending on $\Ad(s)$ is necessary. Details of this construction can be found in \cite[Section~7.1]{Dob19}. 
    
    \begin{lemma}[{\cite[Lemma~7.6]{Dob19}}] \label{Lemma:Braid_Act_k}
        Under the action $\psi_s \circ \textup{Ad} \circ \psi_s^{-1}$ the subgroup $Br(W_X) \times Br(\widet{W})$ maps $\lie{k}$ to itself.
    \end{lemma}


\subsection{Quantum symmetric pairs of type AIII/AIV} \label{Sec:QSPs_AIII}
Let $\field$ be a field of char\-acteristic zero and $q$ an indeterminate. Denote by $\field(q^{1/2})$ be the field of rational functions in $q^{1/2}$ with coefficients in $\field$. 
Following \cite{b-Jantzen96} and \cite{b-Lus94} the Drinfeld-Jimbo quantised enveloping algebra $\Uqg = U_q(\lie{sl}_{n+1}(\C))$ is the associative $\field(q^{1/2})$-algebra generated by elements $E_i, F_i, K_{\mu}$ for $i \in I$ and $\mu \in P$ satisfying the following relations:
    \begin{enumerate}
        \item $K_0 = 1$, $K_{\mu}K_{\lambda}= K_{\mu + \lambda}$ for all $\mu, \lambda \in P$.
        \item $K_{\mu}E_i = q^{(\alpha_i, \mu)}E_iK_{\mu}$ for all $i \in I$, $\mu \in P$.
        \item $K_{\mu}F_i = q^{-(\alpha_i, \mu)}F_iK_{\mu}$ for all $i \in I$, $\mu \in P$.
        \item $E_iF_j - F_jE_i = \delta_{ij}\tfrac{K_i - K_i^{-1}}{q-q^{-1}}$ for all $i,j \in I$.
        \item Quantum Serre relations,
    \end{enumerate}

We use the notation $K_i = K_{\alpha_i}$ for $i \in I$ and $K_{\mu}^{-1} = K_{-\mu}$ for $\mu \in P$ throughout. We make the quantum Serre relations (5) more explicit. Let $p$ denote the non-commutative polynomial in two variables given by
    \begin{equation} \label{Eqn:polynomial}
        p(x,y) = x^2y - (q+q^{-1})xyx + yx^2.
    \end{equation}
Then the quantum Serre relations can be written as
    \begin{align}
        E_iE_j = E_jE_i, \quad F_iF_j = F_jF_i \quad &\mbox{if $a_{ij}=0$,} \label{Eqn:QSerre1}\\
        p(E_i,E_j) = p(F_i,F_j) = 0 \quad &\mbox{if $a_{ij}=-1$.} \label{Eqn:QSerre2}
    \end{align}

Analogously to \eqref{Eqn:Steinberg_Action}, there exists an action of the braid group $Br(\lie{g})$ on $\Uqg$ by algebra automorphisms, see \cite[39.4.3]{b-Lus94}. Under this action the generator $\varsigma_i \in Br(\lie{g})$ is mapped to the Lusztig automorphism $T_i$ as in \cite[Section~8.14]{b-Jantzen96}. We recall explicitly how $T_i$ acts on the generators of $\Uqg$. For any $a,b \in \Uqg$, $c \in \field(q^{1/2})$ let
    \begin{equation} \label{Eqn:Bracket}
        [a,b]_c = ab - cba
    \end{equation}
For any $i, j \in I$ and $\mu \in P$ we have
    \begin{equation} \label{Eqn:Ti(sl2)}
        T_i(E_i) = -F_iK_i, \quad T_i(F_i) = -K_i^{-1}E_i,
    \end{equation}
and
    \begin{align}
        T_i(K_{\mu}) 
            &= K_{\sigma_i(\mu)},  \label{Eqn:TiKj}\\
        T_i(E_j)
            &=  \begin{cases}
                    E_j & \mbox{if $a_{ij} = 0$,}\\
                    [E_i, E_j]_{q^{-1}} & \mbox{if $a_{ij} = -1$,}
                \end{cases} \label{Eqn:TiEj} \\
        T_i(F_j)
            &=  \begin{cases}
                    F_j & \mbox{if $a_{ij} = 0$,}\\
                    [F_j, F_i]_{q} & \mbox{if $a_{ij} = -1$,}
                \end{cases}\label{Eqn:TiFj}
    \end{align}
For any $w \in W$ with reduced expression $w = \sigma_{i_1}\sigma_{i_2} \dotsm \sigma_{i_t}$ we write
    \begin{equation} \label{Eqn:Tw}
        T_w := T_{i_1}T_{i_2} \dotsm T_{i_t}.
    \end{equation}
 For any $J= \{a, a+1,\dotsc, b-1, b\} \subset I$ with $a < b$ we define elements
\begin{align}
E_J^+ &:= \big[ E_a, [E_{a+1}, \dotsc , [E_{b-1}, E_b]_{q^{-1}} \dotsc ]_{q^{-1}}\big]_{q^{-1}}, \label{Eqn:EJ+}\\
E_J^- &:= \big[ E_b, [E_{b-1}, \dotsc , [E_{a+1}, E_a]_{q^{-1}} \dotsc ]_{q^{-1}} \big]_{q^{-1}} \label{Eqn:EJ-}
\end{align}
and similarly we write
\begin{align}
F_J^+ &:= \big[ F_a, [F_{a+1}, \dotsc, [F_{b-1}, F_b]_q \dotsc ]_q\big]_q, \label{Eqn:FJ+}\\
F_J^- &:= \big[F_b, [F_{b-1}, \dotsc, [F_{a+1}, F_a]_q \dotsc ]_q\big]_q. \label{Eqn:FJ-}
\end{align}
Additionally, let
\begin{equation} \label{Eqn:KJ}
K_J = K_aK_{a+1} \dotsm K_{b-1}K_b.
\end{equation}
If $J = \{a\} \subset I$ then we write
\begin{equation} \label{Eqn:EJFJ-oneelt}
E_J^+ = E_J^- = E_a, \quad F_J^+ = F_J^- = F_a.
\end{equation}
For later use, we note the following formulae, which follow from \cite[Lemma~3.4]{a-Kolb14}. We have
\begin{align}
T_{w_X}(F_X^{+}) &= -K_X^{-1}E_X^{+}, & T_{w_X}(F_X^-) &= -K_X^{-1}E_X^-, \label{Eqn:TwXFX} \\
T_{w_X}(E_X^{+}) &= -F_X^{+}K_X, & T_{w_X}(E_X^-) &= -F_X^-K_X. \label{TwXEX}
\end{align}

Following \cite{a-Letzter99a} and the conventions of \cite{a-Kolb14} we now recall the definition of quantum symmetric pair coideal subalgebras for Satake diagrams $(X,\t)$ of type AIII/AIV. Let $\mathcal{M}_X = U_q(\lie{g}_X)$ denote the subalgebra of $\Uqg$ generated by $\{E_i, F_i, K_i^{\pm 1} \mid i \in X\}$. Let $U_{\Theta}^0$ = $\field(q^{1/2})\langle K_{\mu} \mid \mu \in P, -w_X \circ \t(\mu) = \mu \rangle$. By construction, $K_i \in U_{\Theta}^0$ for $i \in X$ and $K_{\varpi_i- \varpi_{\t(i)}} \in U_{\Theta}^0$ for $i \in I$. 
We use the notation 
\begin{equation} \label{Eqn:varpi'}
\varpi_i' = \varpi_i -\varpi_{\t(i)} \quad \mbox{for $i \in I$.} 
\end{equation}
Quantum symmetric pair coideal subalgebras depend on a choice of parameters $\be = (\cs_i)_{i \in I \setminus X} \in (\field(q^{1/2})^{\times})^{I \setminus X}$ satisfying additional constraints. We assume for the remainder of this paper that
    \begin{equation} \label{Eqn:cs_cond}
        \cs_i = \cs_{\t(i)} \quad \mbox{for $i \in I \setminus X \cup \{r,\t(r)\}$,}
    \end{equation}
compare with \cite[Section~5.1]{a-Kolb14}.

\begin{remark}
	The parameters $\cs_i$ relate to the parameters $c_i$ and $s(i)$ seen in \cite{a-Kolb14}, \cite{a-BK15} and \cite{a-BK19} in the following way: $\cs_i = c_is(\t(i))$ for each $i \in I \setminus X$.  
\end{remark}

Following \cite[Definition~5.1, 5.6]{a-Kolb14} we denote by $\Beps = \Beps(X,\t)$ the subalgebra of $\Uqg$ generated by $\mathcal{M}_X$, $U_{\Theta}^0$ and the elements
    \begin{equation} \label{Eqn:Bi_AIII}
        B_i = 
            \begin{cases}
                F_i - \cs_iE_{\t(i)}K_i^{-1} & \mbox{if $i \neq r, \t(r)$,}\\
                F_r - \cs_r[E_X^+, E_{\t(r)}]_{q^{-1}}K_r^{-1} & \mbox{if $i=r$,}\\
                F_{\t(r)} - \cs_{\t(r)}[E_X^-, E_r]_{q^{-1}}K_{\t(r)}^{-1} & \mbox{if    $i=\t(r)$.} 
            \end{cases}
    \end{equation}
for all $i \in I \setminus X$. For consistency, we set $B_i = F_i$ and $\cs_i = 0$ for $i \in X$.

We recall now the defining relations of $\Beps$, following \cite[Section~7]{a-Letzter03} and \cite[Section~7]{a-Kolb14}. For $i \in I \setminus X$ let
    \begin{equation}
        L_i = K_iK_{\t(i)}^{-1}
    \end{equation}
and define
    \begin{equation} \label{Eqn:Zi}
        \mathcal{Z}_i =  
            \begin{cases}
                -(1-q^{-2})E_X^+\K{\t(r)} & \mbox{if $i=r$,}\\
                -(1-q^{-2})E_X^-\K{r} & \mbox{if $i=\t(r)$,}\\
                -\K{\t(i)} & \mbox{otherwise.}
            \end{cases}
    \end{equation}
Further, let
    \begin{equation} \label{Eqn:Gi}
        \G_i:= \cs_i\Z_i - \cs_{\t(i)}\Z_{\t(i)}
    \end{equation}
for $i \in I \setminus X$.
Then the algebra $\Beps$ is generated over $\mathcal{M}_X\Uow$ by the elements $B_i$ for $i \in I \setminus X$, subject to the relations
    \begin{align}
        B_iK_{\mu} &= q^{(\mu,\alpha_i)}K_{\mu}B_i
            & &\mbox{for $i \in I \setminus X$, $K_{\mu} \in U_{\Theta}^0$,} \label{Eqn:BiRel1}\\
        B_iE_j &= E_jB_i
            & &\mbox{for $i \in I \setminus X$, $j \in X$,} \label{Eqn:BiRel2}\\
        B_iB_j - B_jB_i &= \delta_{i, \t(j)}(q-q^{-1})^{-1}\G_i
            & &\mbox{for $i \in I \setminus X$, $j \in I$, $a_{ij}=0$,} \label{Eqn:BiRel3}\\
       p(B_i,B_j) &= 0
            & &\mbox{for $i, j \in I$, $a_{ij} = -1$.}\label{Eqn:BiRel4}
    \end{align}
%
    

\section{Main Results} \label{Sect:Braid_Action}
Recall from Lemma \ref{Lemma:Braid_Act_k} that an action of $Br(W_X) \times Br(\widet{W})$ on $\lie{k}$ by Lie algebra automorphisms is obtained by restriction of the action of $Br(\lie{g})$ on $\lie{g}$. We now construct an analogous braid group action in the setting of quantum symmetric pairs of type AIII/AIV. Recall that the algebra automorphisms $T_i$ for $i \in X$ give rise to a representation of $Br(W_X)$ on $\Uqg$. In \cite[Section~4.1]{a-BW16} it was shown that $\Beps$ is invariant under the automorphisms $T_i$ for $i \in X$. 

    \begin{theorem}[{\cite[Section~4.1]{a-BW16}}] \label{Thm:WX_Act}
        There exists an action of $Br(W_X)$ on $\Beps$ by algebra automorphisms such that the generator $\varsigma_i \in Br(W_X)$ is mapped to the Lusztig automorpshism $T_i$.
    \end{theorem}

We give this action explicitly on the elements $B_i$ for $i \in I \setminus X$. For $i \in X$ and $j \in I \setminus X$ we have
    \begin{equation} \label{Eqn:TiBj}
        T_i(B_j) =
            \begin{cases}
                B_j & \mbox{if $a_{ij} = 0$,}\\
                [B_j, F_i]_q & \mbox{if $a_{ij}=-1$,}
            \end{cases}
           \quad
         T_i^{-1}(B_j) =
         	\begin{cases}
	         	B_j & \mbox{if $a_{ij} = 0$,}\\
	         	[F_i, B_j]_q & \mbox{if $a_{ij} = -1$.}
         	\end{cases}
    \end{equation}
It follows that
\begin{align}
    T_{w_X}(B_r) &= [B_r, F_X^+]_q, \label{Eqn:TwxBr}\\
    T_{w_X}(B_{\t(r)}) &= [B_{\t(r)}, F_X^-]_q. \label{Eqn:TwxBtr}
\end{align}
Similarly, one also obtains
\begin{align}
    T_{w_X}^{-1}(B_r) &= [F_X^-, B_r]_q, \label{Eqn:TwxiBr}\\
    T_{w_X}^{-1}(B_{\t(r)}) &= [F_X^+, B_{\t(r)}]_q. \label{Eqn:TwxiBtr}
\end{align}
We now construct the action of $Br(\widet{W})$ on $\Beps$ by algebra automorphisms. 
For reasons observed in \cref{Eqn:cti,Eqn:ctr} we now consider an extension $\field'$ of the field $\field(q^{1/2})$ that contains $\sqrt{\cs_i}$ for $i \in I \setminus X$.
For $1 \leq i \leq r$ the algebra automorphisms
    \begin{equation} \label{Eqn:wideTi}
        \widet{T}_i := T_{\widet{\sigma}_i}
            =   \begin{cases}
                    T_iT_{\t(i)} & \mbox{if $1 \leq i \leq r$,}\\
                    T_rT_{r+1} \dotsm T_{\t(r)} \dotsm T_{r+1}T_r & \mbox{if $i =r$}
                \end{cases}
    \end{equation}
do not leave $\Beps$ invariant. However, they are still used as a guide to the construction of a braid group action on $\Beps$. The general strategy is similar to that of \cite{a-KP11}. We first define the action of the generators $\widet{\varsigma}_i$ for $1 \leq i \leq r-1$.

For $1 \leq i \leq r-1$ and $j \in I \setminus X$ define
    \begin{equation} \label{Eqn:cti}
        \ct_i(B_j)
            =   \begin{cases}
                    q^{-1}B_{\t(j)}\K{\t(j)} & \mbox{if $j = i$ or $j = \t(i)$,}\\
                    \big(q\cs_i \big)^{-1/2} [B_j, B_i]_q & \mbox{if $a_{ij}=-1$,}\\
                    \big(q\cs_{\t(i)} \big)^{-1/2} [B_j, B_{\t(i)}]_q & \mbox{if $a_{\t(i)j} = -1$,}\\
                    B_j & \mbox{if $a_{ij}=0$ and $a_{\t(i)j}=0$.}
                \end{cases}
    \end{equation}

\begin{theorem} \label{Thm:cti}
     Suppose $(X, \t)$ is a Satake diagram of type AIII with $X = \{r+1, \dotsc \t(r+1)\}$ and $1 \leq r \leq \lceil \frac{n}{2} \rceil - 1$. Let $1 \leq i \leq r-1$.
        \begin{enumerate} [label = \textup{(\arabic*)}]
            \item There exists a unique algebra automorphism $\ct_i$ of $\Beps$ such that $\ct_i(B_j)$ is given by \cref{Eqn:cti} for $j \in I \setminus X$ and $\ct_i|_{\mathcal{M}_X\Uow} = \widet{T}_i|_{\mathcal{M}_X\Uow}$.
            \item The inverse automorphism $\ct_i^{-1}$ is given by 
                \begin{equation} \label{Eqn:cti_inv}
                    \ct_i^{-1}(B_j)
                        =   \begin{cases}
                                qB_{\t(j)}\K{j} & \mbox{if $j = i$ or $j = \t(j)$,}\\
                                \big(q\cs_i \big)^{-1/2} [B_i, B_j]_q & \mbox{if $a_{ij}=-1$,}\\
                                \big(q\cs_{\t(i)} \big)^{-1/2} [B_{\t(i)}, B_{j}]_q & \mbox{if $a_{\t(i)j} = -1$,}\\
                                B_j & \mbox{if $a_{ij}=0$ and $a_{\t(i)j}=0$}
                            \end{cases}
                \end{equation}
            and $\ct_i^{-1}|_{\mathcal{M}_X\Uow} = \widet{T}_i|_{\mathcal{M}_X\Uow}$.
            \item The relation $\ct_i\ct_{i+1}\ct_i = \ct_{i+1}\ct_i\ct_{i+1}$ holds for $1 \leq i < r-1$. Further the relation $\ct_i\ct_j = \ct_j\ct_i$ holds for $a_{ij}=0$ with $1 \leq i,j \leq r-1$.
        \end{enumerate}
\end{theorem}

    \begin{proof}
    The result follows from \cite[Theorems~4.3 and 4.6]{a-KP11} where the only difference occurs in $\ct_i(B_j)$ when $a_{ij}=-1$ or $a_{\t(i)j}=-1$. Here, one checks that
        \begin{align*}
            [B_j, B_i]_q[B_{\t(j)},B_{\t(i)}]_q &- [B_{\t(j)},B_{\t(i)}]_q[B_j,B_i]_q\\ &=\frac{q\cs_i}{q-q^{-1}}\big(\cs_j\ct_i(\Z_j)-\cs_{\t(j)}\ct_i(\Z_{\t(j)}) \big).
        \end{align*}
    Hence for symmetry reasons and the fact that $\cs_i = \cs_{\t(i)}$ for $1 \leq i \leq r-1$, we choose $\ct_i(B_j)$ and $\ct_i(B_{\t(j)})$ as in \cref{Eqn:cti}.
    \end{proof}

\begin{remark}
In \cref{Eqn:cti,Eqn:cti_inv} the coefficients $\cs_i$ appear whereas they did not in \cite{a-KP11}. This is because Kolb and Pellegrini took $\cs_i = 1$ for all $i \in I$ in their paper. 
\end{remark}

It remains to construct the algebra automorphism $\ct_r$. For ease of notation, let $C = \big(q\cs_r\cs_{\t(r)}\big)^{-1/2}$. Recall from Equation \eqref{Eqn:varpi'} that we set $\varpi'_i = \varpi_i - \varpi_{\t(i)}$ for $i \in I$. Define
    \begin{equation} \label{Eqn:ctr}
        \ct_r(B_j)
            =   \begin{cases}
                    q^{-1}B_r\K{r}K_{\varpi'_{r+1}} &\mbox{if $j=r$,}\\
                    q^{-1}B_{\t(r)}\K{\t(r)}K_{\varpi'_{\t(r+1)}} &\mbox{if $j = \t(r)$,}\\
                    C\big( \big[B_{r-1}, [B_r,[F_X^+, B_{\t(r)}]_q]_q\big]_q\\ 
                        \quad{}+ q\cs_{\t(r)}B_{r-1}\K{r}K_X^{-1}\big) & \mbox{if $j = r-1$,}\\
                    C\big( \big[B_{\t(r-1)}, [B_{\t(r)},[F_X^-, B_{r}]_q]_q\big]_q\\ 
                        \quad{}+ q\cs_rB_{\t(r-1)}\K{\t(r)}K_X^{-1}\big) & \mbox{if $j = \t(r-1)$,}\\
                    B_j & \mbox{otherwise}
                \end{cases}
    \end{equation}
for $j \in I \setminus X$. The following theorem establishes that \cref{Eqn:ctr} defines an algebra automorphism $\ct_r:\Beps \rightarrow \Beps$.

    \begin{theorem} \label{Thm:ctr_Aut}
        Suppose $(X, \t)$ is a Satake diagram of type AIII/AIV with $X = \{r+1, \dotsc \t(r+1)\}$ and $1 \leq r \leq \lceil \frac{n}{2} \rceil - 1$.
        \begin{enumerate} [label = \textup{(\arabic*)}]
            \item There exists a unique algebra automorphism $\ct_r$ of $\Beps$ such that $\ct_r(B_j)$ is given by \cref{Eqn:ctr} and $\ct_r|_{\mathcal{M}_X\Uow} = \widet{T}_r|_{\mathcal{M}_X\Uow}$.
            \item The inverse automorphism $\ct_r^{-1}$ is given by
                \begin{equation} \label{Eqn:ctri}
                \ct_r^{-1}(B_j)
                    =   \begin{cases}
                        qB_r\K{\t(r)}K_{\varpi'_{\t(r+1)}} &\mbox{if $j=r$,}\\
                        qB_{\t(r)}\K{r}K_{\varpi'_{r+1}} &\mbox{if $j = \t(r)$,}\\
                        C\big( \big[B_{\t(r)}, [F_X^-,[B_r, B_{r-1}]_q]_q\big]_q\\ 
                            \quad{}+ \cs_{r}B_{r-1}\K{\t(r)}K_X^{-1}\big) & \mbox{if $j = r-1$,}\\
                        C\big( \big[B_{r}, [F_X^+,[B_{\t(r)}, B_{\t(r-1)}]_q]_q\big]_q\\ 
                            \quad{}+ \cs_{\t(r)}B_{\t(r-1)}\K{r}K_X^{-1}\big) & \mbox{if $j = \t(r-1)$,}\\
                        B_j & \mbox{otherwise}
                    \end{cases}
                \end{equation}
        \end{enumerate}
        and $\ct_r^{-1}|_{\mathcal{M}_X\Uow}= \widet{T}_r^{-1}|_{\mathcal{M}_X\Uow}$.
    \end{theorem}

    \begin{remark} \label{Rem:Why_Large}
        A desirable property of the algebra automorphism $\ct_r$ is that it is local, meaning $\ct_r(B_i) = B_i$ is satisfied for $i \in I \setminus X$ and $a_{ir} = a_{\t(i)r} = 0$. With this, it is not possible to omit the elements $K_{\varpi_i'}$ for $i \in I \setminus X$ from our constructions. In particular, $K_{\varpi_i'}$ appears so that the relation
        \[ B_iE_j - E_jB_i = 0 \]
        for $i \in I \setminus X$, $j \in X$ is preserved under $\ct_r$.   
    \end{remark}

The proof of \cref{Thm:ctr_Aut} requires non-trivial calculations which are postponed to \cref{Sec:Proof1}. Crucially, the algebra automorphisms $\ct_1, \dotsc, \ct_r$ satisfy type $\text{B}_r$ braid relations.

    \begin{theorem} \label{Thm:ctr_braid}
        Suppose $(X,\t)$ is a Satake diagram of type AIII with $X= \{r+1, \dotsc, \t(r+1)\}$ and $1 \leq r \leq \lceil \frac{n}{2} \rceil - 1$. Then the relation
            \begin{equation} \label{Eqn:ctr_Braid_Rel}
                \ct_r\ct_{r-1}\ct_r\ct_{r-1}= \ct_{r-1}\ct_r\ct_{r-1}\ct_r
            \end{equation}
        holds. Further, the relations $\ct_r\ct_i = \ct_i\ct_r$ hold for any $1 \leq i < r-1$.
    \end{theorem}

Similarly to \cref{Thm:ctr_Aut}, the proof of \cref{Thm:ctr_braid} requires a series of calculations which are given in \cref{Sec:Proof2}. As a result of \cref{Thm:cti,,Thm:ctr_Aut,Thm:ctr_braid} a braid group action of $Br(\widet{W})$ on $\Beps$ by algebra automorphisms is established.

    \begin{corollary} \label{Cor:Wtild_Act}
        Suppose $(X, \t)$ is a Satake diagram of type AIII/AIV with $X = \{r+1,\dotsc, \t(r+1)\}$ and $1 \leq r \leq \lceil \frac{n}{2} \rceil - 1$. Then there exists an action of $Br(\widet{W})$ on $\Beps$ by algebra automorphisms. Under this action the generator $\widet{\varsigma}_i \in Br(\widet{W})$ is mapped to the algebra automorphism $\ct_i$ for $1 \leq i \leq r$.
    \end{corollary}

Since the subgroups $Br(W_X)$ and $Br(\widet{W})$ of $Br(\lie{g})$ commute, we now combine \cref{Thm:WX_Act,Cor:Wtild_Act} to give an action of $Br(W_X) \times Br(\widet{W})$ on $\Beps$ by algebra automorphisms.

    \begin{theorem} \label{Thm:MAIN}
        Let $(X, \t)$ be a Satake diagram of type AIII/AIV with $X = \{r+1,\dotsc, \t(r+1)\}$ and $1 \leq r \leq \lceil \frac{n}{2} \rceil -1$. Then there exists an action of $Br(W_X) \times Br(\widet{W})$ on $\Beps$ by algebra automorphisms. The action of $Br(W_X)$ on $\Beps$ is given by the Lusztig automorphisms $T_j$ for $j \in X$ and the action of $Br(\widet{W})$ on $\Beps$ is given by the algebra automorphisms $\ct_i$ for $1 \leq i \leq r$ given by \cref{Eqn:cti,Eqn:ctr}.
    \end{theorem}
    
\begin{proof}
In order to prove \cref{Thm:MAIN} it suffices to show that for all $x \in \Beps$, the relation \[T_j\ct_i(x) = \ct_iT_j(x)\] holds for $j \in X$ and $1 \leq i \leq r$. Since $\ct_i|_{\mathcal{M}_X\Uow} = \widet{T}_i|_{\mathcal{M}_X\Uow}$ for all $1 \leq i \leq r$ the result follows if $x \in M_XU_0^{\Theta}$. As a result of this and the underlying symmetry, we hence only consider $x = B_k$ for $1 \leq k \leq r$. We proceed by casework. 
 
\phantom{m}

\noindent \textbf{Case 1.} $1 \leq i \leq r-1$ and $j \in X \setminus \{r+1\}$.
\newline 
By \cref{Eqn:TiBj,Eqn:cti} it follows that $\ct_i(B_k)$ is invariant under $T_j$ for all $j \in X \setminus \{r+1\}$ and $1 \leq k \leq r.$ The result follows from this.
        
\phantom{new line}
        
\noindent \textbf{Case 2.} $1 \leq i \leq r-1$ and $j = r+1$.
\newline         
Recall from \cref{Eqn:TiBj} that for $1 \leq k \leq r$ we have
	\begin{equation*} 
		T_{r+1}(B_k) 
			=   \begin{cases}
				B_k & \mbox{if $1 \leq k \leq r-1$,}\\
				[B_r,F_{r+1}]_q & \mbox{if $k = r$.}
				\end{cases}
	\end{equation*}
There are three cases to consider, depending on the value of $a_{ik}$. If $a_{ik} = 0$ then $\ct_i(B_k) = B_k$ and $\ct_i \circ T_{r+1}(B_k) = T_{r+1}(B_k)$. The claim follows from this. If $a_{ik} = -1$ then $\ct_i(B_k) = \big(q\cs_i \big)^{-1/2}[B_k, B_i]_q$. Since $1 \leq i \leq r-1$, \cref{Eqn:TiBj} implies that we need only check the claim when $k = r$ and $i = r-1$. We obtain
	\begin{align*}
		T_{r+1}\ct_{r-1}(B_r)
			&= \big( q\cs_i \big)^{-1/2} T_{r+1}([B_r,B_{r-1}]_q)\\
			&= \big( q\cs_i \big)^{-1/2} [[B_r,F_{r+1}]_q,B_{r-1}]_q\\
			&= \big( q\cs_i \big)^{-1/2}[[B_r,B_{r-1}]_q,F_{r+1}]_q = \ct_{r-1}T_{r+1}(B_r)
	\end{align*}
as required. Finally, if $a_{ik} = 2$ then the claim follows since $1 \leq i \leq r-1$ and hence $\ct_i(B_i) = q^{-1}B_{\t(i)}\K{\t(i)}$ is invariant under $T_{r+1}$.

\phantom{new line}

\noindent \textbf{Case 3.} $i = r$ and $k \neq r-1$.
\newline 
Suppose that $1 \leq k \leq r-2$. Then both $T_j$ and $\ct_r$ act as the identity on $B_k$ so the claim follows. Hence assume that $k = r$. Then recall from \cref{Eqn:ctr} that
	\begin{equation*}
		\ct_r(B_r) = q^{-1}B_r\K{r}K_{\varpi'_{r+1}}
	\end{equation*}
where $\varpi'_{r+1} = \varpi_{r+1} - \varpi_{\t(r+1)}$. Let $\lambda = \alpha_r - \alpha_{\t(r)} +\varpi'_{r+1}$. Since $\alpha_r = -\varpi_{r+1} + 2\varpi_r - \varpi_{r-1}$ it follows that $(\alpha_j, \lambda) = 0$ for all $j \in X$. This implies that $\sigma_j(\lambda) = \lambda$ for all $j \in X$ and hence $T_j(\K{r}K_{\varpi'_{r+1}})=\K{r}K_{\varpi'_{r+1}}$.

If $j \neq r+1$ then $T_j(B_r) = B_r$ and the result follows. Otherwise by Equation \eqref{Eqn:TiBj} we have
	\begin{align*}
		T_{r+1}\ct_r(B_r) &= q^{-1}T_{r+1}(B_r)K_{\lambda}\\
			&=q^{-1}[B_r,F_{r+1}]_qK_{\lambda} = \ct_rT_{r+1}(B_r)
	\end{align*}
where we use the fact that $K_{\lambda}$ commutes with $F_{r+1}$. The result follows from this.

\phantom{new line}

\noindent \textbf{Case 4.} $i =r$, $j \in X \setminus \{r+1\}$ and $k = r-1$.
\newline
By \cref{AppLem:6} the result is clear for $j \neq \t(r+1)$ since $T_j$ acts as the identity on the elements $F_X^+, \K{r}, K_X^{-1}$ and $B_k$ for $k \in I \setminus X$. On the other hand if $j = \t(r+1)$ then by Equation \eqref{Eqn:TiBj} we have
	\begin{align*}
		T_{\t(r+1)}([F_X^+,B_{\t(r)}]_q)
			&= T_{\t(r+1)}\big([F_{X\setminus \{\t(r+1)\}}, [F_{\t(r+1)},B_{\t(r)}]_q]_q\big)\\
			&= [F_X^+, B_{\t(r)}]_q.
	\end{align*}
Hence $\ct_r(B_{r-1})$ is invariant under $T_{\t(r+1)}$ and the result follows.

\phantom{new line}

\noindent \textbf{Case 5.} $i =r$, $j =r+1$ and $k = r-1$.
\newline
Recall from \cref{Eqn:ctr} that
	\begin{equation*}
		\ct_r(B_{r-1}) = C\big( \big[B_{r-1},[B_r, [F_X^+, B_{\t(r)}]_q]_q\big]_q + q\cs_{\t(r)}B_{r-1}\K{r}K_X^{-1}\big).
	\end{equation*}
We are done if we show that $\ct_r(B_{r-1})$ is invariant under $T_{r+1}$. Using Lemma \ref{AppLem:7} we obtain
	\begin{align*}
		T_{r+1}\big(\big[B_{r-1}, [B_r, [F_X^+,B_{\t(r)}]_q]_q\big]_q\big)
			&= \big[B_{r-1}, [B_r, [F_X^+,B_{\t(r)}]_q]_q\big]_q\\
			&\quad{}+ q\cs_{\t(r)}B_{r-1}\K{r}(K_{r+1}^{-1}- K_{r+1})K_{X \setminus \{r+1\}}^{-1}.
	\end{align*}
Further, we have
	\begin{equation*}
		T_{r+1}(B_{r-1}\K{r}K_{X}^{-1}) = B_{r-1}\K{r}K_{r+1}K_{X \setminus \{r+1\}}^{-1}.
	\end{equation*}
Combining these we obtain
	\begin{align*}
		T_{r+1}\ct_r(B_{r-1}) 
			&= C\big[B_{r-1}, [B_r, [F_{X},B_{\t(r)}]_q]_q\big]_q + Cq\cs_{\t(r)}B_{r-1}\K{r}K_{r+1}K_{X \setminus \{r+1\}}^{-1}\\
			&\quad{} +  Cq\cs_{\t(r)}B_{r-1}\K{r}(K_{r+1}^{-1}-K_{r+1})K_{X \setminus \{r+1\}}^{-1}\\
			&= \ct_r(B_{r-1})
	\end{align*}
as required. 
\end{proof}


\section{Proof of \texorpdfstring{\cref{Thm:ctr_Aut}}{Theorem1}} \label{Sec:Proof1}
We divide the proof of Theorem \ref{Thm:ctr_Aut} into three parts. In the first part, we show that $\ct_r$ is an algebra endomorphism of $\Beps$ by checking that all of the necessary relations are satisfied. Next, we show that $\ct_r^{-1}$ is also an algebra endomorphism of $\Beps$. Finally, we show that $\ct_r^{-1}$ really is the inverse of $\ct_r$.

\subsection{Proof that $\ct_r$ is an algebra endomorphism} \label{Subsec:ctr_endo}
Recall from \cref{Eqn:polynomial,Eqn:Gi} the poly\-nomial $p:\Uqg \times \Uqg \rightarrow \Uqg$ and the elements $\G_i = \cs_i\Z_i - \cs_{\t(i)}\Z_{\t(i)}$. In view of relations \eqref{Eqn:BiRel1} to \eqref{Eqn:BiRel4} and \cref{Eqn:ctr} we show that the relations
    \begin{align}
        B_{r-1}x - xB_{r-1} &= 0 \quad \mbox{for $x \in \mathcal{M}_X$,} \label{Eqn:ctrCheck1}\\
        B_{r-1}B_{\t(r-1)} - B_{\t(r-1)}B_{r-1} &= \frac{1}{q-q^{-1}}\G_{r-1}, \label{Eqn:ctrCheck2}\\
        p(B_r,B_{r-1}) &= 0, \label{Eqn:ctrCheck3}\\
        p(B_{r-1},B_r) &= 0 \label{Eqn:ctrCheck4}
    \end{align}
are preserved under the map $\ct_r$. The remaining relations either follow from the above by symmetry, or can be verified by short calculations. Such checks are not shown here. Using \cref{AppLem:Alt_ctr}, the elements $\ct_r(B_{r-1})$ and $\ct_r(B_{\t(r-1)})$ can be expressed in the following way. Let
    \begin{align}
        S &= \big[B_{r-1}, [B_r, F_X^+]_q \big]_q, \label{Eqn:S} \\ 
        S^{\t} &= \big[B_{\t(r-1)}, [B_{\t(r)}, F_X^-]_q \big]_q \label{Eqn:St} 
    \end{align}
and let
    \begin{align}
        \Delta &= q\cs_{\t(r)}B_{r-1}\K{r}K_X, \label{Eqn:D} \\
        \Delta^{\t} &= q\cs_{r}B_{\t(r-1)}\K{\t(r)}K_X. \label{Eqn:Dt}
    \end{align}
Then we have
    \begin{align}
        \ct_r(B_{r-1}) &= C\big([S,B_{\t(r)}]_q + \Delta \big), \label{Eqn:ctr(Br-1)_NEW}\\
        \ct_r(B_{\t(r-1)}) &= C\big([S^{\t},B_r]_q + \Delta^{\t} \big). \label{Eqn:ctr(Btr-1)_NEW}
    \end{align}
We use \cref{Eqn:ctr(Br-1)_NEW,Eqn:ctr(Btr-1)_NEW} to establish many of the results of this section.
We first show that \eqref{Eqn:ctrCheck1} is invariant under $\ct_r$. The following lemma establishes invariance for $x \in \{E_i, F_i \mid \mbox{$i \in X$, $i \neq r+1, \t(r+1)$} \}$.

    \begin{lemma} \label{Lemma:ctrCheck1-Part1}
        For any $i \in X\setminus \{r+1, \t(r+1) \}$ the relations
            \begin{align*}
                F_X^+E_i - E_iF_X^+ &= 0,\\
                F_X^+F_i - F_iF_X^+ &= 0
            \end{align*}
        hold in $\Uqg$. 
    \end{lemma}

    \begin{proof}
        For any $i \in X \setminus \{r+1, \t(r+1)\}$ let $W_i = \{r+1, r+2, \dotsc, i-1\}$ and $Y_i = \{i+1,i+2, \dotsc, \t(r+1)\}$. Since $E_iF_j - F_jE_i = \delta_{ij}(q-q^{-1})^{-1}(K_i-K_i^{-1})$ for $i, j \in I$, it follows that
            \begin{equation*}
                F_X^+E_i = E_iF_X^+ - \frac{1}{q-q^{-1}}\big[ F_{W_i}^+, [K_i - K_i^{-1}, F_{Y_i}^+]_q  \big]_q.
            \end{equation*}
        Since $[K_i, F_{i+1}]_q = 0$ and $[F_{i-1},K_i^{-1}]_q =0$, it follows that $F_X^+E_i = E_iF_X^+$. Now, using Equation \eqref{Eqn:TiFj} we have
            \begin{align*}
                F_i &= T_{i-1}^{-1}T_i^{-1}(F_{i-1}),\\
                \big[F_{i-1}, [F_i, F_{i+1}]_q\big]_q &= T_{i-1}^{-1}T_i^{-1}(F_{i+1}).
            \end{align*}
        This implies that
        \[ \big[F_{i-1}, [F_i,F_{i+1}]_q \big]_qF_i = F_i\big[F_{i-1}, [F_i,F_{i+1}]_q \big]_q \]
        and hence $F_X^+F_i =F_iF_X^+$ as required.
    \end{proof}

We now consider \eqref{Eqn:ctrCheck1} for $x \in \{E_{r+1}, E_{\t(r+1)}, F_{r+1}, F_{\t(r+1)}\}$.

    \begin{lemma} \label{Lemma:ctrCheck1-Part2}
        The relations 
            \begin{align*}        
                \ct_r(B_{r-1})E_{r+1}-E_{r+1}\ct_r(B_{r-1})&= 0,\\
                \ct_r(B_{r-1})E_{\t(r+1)} - E_{\t(r+1)}\ct_r(B_{r-1}) &= 0
            \end{align*}
        hold in $\Beps$.
    \end{lemma}

    \begin{proof}
        Suppose first that $X = \{ r + 1 \}$. Since $[B_r, K_{r+1}^{-1}]_q = 0$ we have 
            \begin{align*}
                S E_{r+1} 
                    &= \big[ B_{r-1}, [ B_r, F_{r+1}E_{r+1} ]_q \big]_q \\
                    &= E_{r+1} S - \frac{1}{q-q^{-1}} \big[ B_{r-1}, [B_r, K_{r+1} ]_q \big]_q \\
                    &= E_{r+1} S + q[ B_{r-1}, B_r ]_q K_{r+1}.
            \end{align*}
        It follows that
            \begin{align*}
                [ S, B_{r+2} ]_q E_{r+1}
                    &= E_{r+1} [ S, B_{r+2} ]_q + q \big[ [ B_{r-1}, B_r ]_q K_{r+1}, B_{r+2} \big]_q\\
                    &= E_{r+1} [ S, B_{r+2} ]_q + q^2\big[ B_{r-1}, [ B_r, B_{r+2}] \big]_q K_{r+1} \\
                    &= E_{r+1} [ S, B_{r+2} ]_q - \frac{ q^2\cs_{r+2} }{q-q^{-1}} [ B_{r-1}, \Z_{r+2} ]_q K_{r+1}.
            \end{align*}
        Recalling that $\Z_{r+2} = -(1-q^{-2}) E_{r+1} \K{r}$ we obtain using \eqref{Eqn:BiRel1} and \eqref{Eqn:BiRel2}
            \[ [ S, B_{r+2} ]_q E_{r+1} = E_{r+1} [ S, B_{r+2} ]_q + q (1-q^2) \cs_{r+2} E_{r+1} B_{r-1} \K{r} K_{r+1}. \]
        This and \cref{Eqn:Dt} implies
            \begin{align*}
                \frac{1}{C} [ \ct_r(B_{r-1}), E_{r+1} ] 
                    &= [ S, B_{r+2} ]_qE_{r+1} + \Delta E_{r+1} - E_{r+1} [ S, B_{r+2} ]_q - E_{r+1} \Delta \\
                    &= (1 - q^2) E_{r+1} \Delta + \Delta E_{r+1} - E_{r+1} \Delta \\
                    &= 0
            \end{align*}
        where the last equality follows since $E_{r+1} \Delta = q^{-2} \Delta E_{r+1}$. This shows that $E_{r+1}$ commutes with $\ct_r(B_{r-1})$ when $|X| = 1$.
        
        Suppose now that $|X| > 1$. Let $Y = X \setminus \{r+1\}$. Then 
            \begin{align*}
                F_X^+E_{r+1} &= E_{r+1}F_X^+ - \frac{1}{q-q^{-1}}[ K_{r+1}-K_{r+1}^{-1}, F_Y^+]_q\\
                    &= E_{r+1}F_X^+ - F_Y^+K_{r+1}^{-1}.
            \end{align*}
        It follows from this and the relation $[B_r, K_{r+1}^{-1}]_q = 0$ that
            \begin{align*}
                SE_{r+1} &=\big[B_{r-1}, [B_r, F_X^+E_{r+1}]_q \big]_q\\
                    &= \big[B_{r-1},[B_r, E_{r+1}F_X^+ - F_Y^+K_{r+1}^{-1}]_q \big]_q\\
                    &=E_{r+1}S - F_Y^+\big[B_{r-1}, [B_r,K_{r+1}^{-1}]_q \big]_q\\
                    &=E_{r+1}S.
            \end{align*}
        Further, we have $\Delta E_{r+1} = E_{r+1}\Delta$ since $E_{r+1}$ commutes with $\K{r}K_X$. This implies that $\ct_r(B_{r-1})$ commutes with $E_{r+1}$. To show that $\ct_r(B_{r-1})$ commutes with $E_{\t(r+1)}$, one proceeds similarly but instead using the form of $\ct_r(B_{r-1})$ given in \cref{Eqn:ctr} and the relation $[K_{\t(r+1)}, B_{\t(r)}]_q = 0$.
    \end{proof}

    \begin{lemma} \label{Lemma:ctrCheck1-Part3}
        The relations
            \begin{align*}        
                \ct_r(B_{r-1})F_{r+1}-F_{r+1}\ct_r(B_{r-1})&= 0,\\
                \ct_r(B_{r-1})F_{\t(r+1)} - F_{\t(r+1)}\ct_r(B_{r-1}) &= 0
            \end{align*}
        hold in $\Beps$.
    \end{lemma}

    \begin{proof}
        Suppose first that $X = \{ r + 1 \}$. Since $B_{r-1}$ commutes with $B_{r+2}$ we have
            \[ [ S, B_{r+2} ]_q = \big[ B_{r-1}, [ [ B_r, F_{r+1} ]_q, B_{r+2} ]_q \big]_q. \]
        We commute $F_{r+1}$ through $[ S, B_{r+2} ]_q$ using the algebra automorphism $T_{w_X} = T_{r+1}$. In particular, by \cref{Eqn:Ti(sl2),Eqn:TiBj} we have
            \begin{align*}
                F_{r+1} &= -T_{r+1}( E_{r+1} K_{r+1} ),\\
                [ B_r, F_{r+1} ]_q &= T_{r+1}( B_r ), \\
                B_{r+2} &= T_{r+1}( [ F_{r+1}, B_{r+2} ]_q ).
            \end{align*}
        It hence follows that
            \begin{equation} \label{Eqn:[SB]F}
                [ S, B_{r+2} ]_q F_{r+1} = -T_{r+1}\big( \big[ B_{r-1}, [ B_r, [ F_{r+1}, B_{r+2} ]_q ]_q \big]_q E_{r+1} K_{r+1} \big).
            \end{equation}
        We consider the right hand side of \cref{Eqn:[SB]F}. We have
            \begin{align*}
                [ B_r, [ F_{r+1}, B_{r+2} ]_q ]_q E_{r+1} K_{r+1} - E_{r+1} &K_{r+1} [ B_r, [ F_{r+1}, B_{r+2} ]_q ]_q \\
                    &= \frac{1}{q-q^{-1}} [ B_r, [ K_{r+1}^{-1}, B_{r+2}]_q ]_q K_{r+1}\\
                    &= -\frac{1}{q-q^{-1}} (\cs_r \Z_r - \cs_{r+2} \Z_{r+2} ).
            \end{align*}
        Substituting this into \cref{Eqn:[SB]F} we obtain
            \begin{align*}
                [ S, B_{r+2} ]_q F_{r+1} 
                    &= F_{r+1} [ S, B_{r+2} ]_q + \frac{1}{q-q^{-1}} T_{r+1}( [B_{r-1}, \cs_r\Z_r - \cs_{r+2}\Z_{r+2} )\\
                    &= F_{r+1} [ S, B_{r+2} ]_q + q\cs_{r+2} B_{r-1} T_{r+1}(\Z_{r+2}).
            \end{align*}
        Since $\Z_{r+2} = -(1-q^{-2})E_{r+1}\K{r}$ we have
            \[ T_{r+1}( \Z_{r+2} ) = (1-q^{-2}) F_{r+1}K_{r+1}\K{r}. \]
        This implies that
            \begin{align*}
                [ S, B_{r+2} ]_q F_{r+1}
                    &= F_{r+1} [ S, B_{r+2} ]_q + q(1-q^{-2}) \cs_{r+2} F_{r+1} B_{r-1} K_{r+1} \K{r} \\
                    &= F_{r+1} [ S, B_{r+2} ]_q + (1-q^{-2}) F_{r+1} \Delta.
            \end{align*}
        Since $F_{r+1} \Delta = q^2 \Delta F_{r+1}$ we have
            \begin{align*}
                \frac{1}{C} [ \ct_r(B_{r-1}), F_{r+1} ] 
                    &= [ S, B_{r+2} ]_q F_{r+1} + \Delta F_{r+1} - F_{r+1} [ S, B_{r+2} ]_q - F_{r+1} \Delta \\
                    &= (1-q^{-2}) F_{r+1} \Delta + \Delta F_{r+1} - F_{r+1} \Delta\\
                    &= 0.
            \end{align*}
        This shows that $F_{r+1}$ commutes with $\ct_r(B_{r-1})$ when $|X|=1$.
        
        Suppose now that $|X| > 1$.
        The relation $F_{r+1} = -T_{w_X}(E_{\t(r+1)}K_{\t(r+1)})$ and Equation \eqref{Eqn:TwxBr} imply 
            \begin{align*}
                \big[ F_{r+1}, [B_r,F_X^+]_q \big] &= -T_{w_X}\big([E_{\t(r+1)}K_{\t(r+1)}, B_r] \big)\\
                    &= 0
            \end{align*}
        and hence $S$ commutes with $F_{r+1}$. This, paired with the relation
            \[ F_{r+1}\K{r}K_X = \K{r}K_XF_{r+1}, \]
        shows that $F_{r+1}$ commutes with $\ct_r(B_{r-1})$. In order to verify that $\ct_r(B_{r-1})$ commutes with $F_{\t(r+1)}$, one shows that $F_{\t(r+1)}$ commutes with $[F_X^+, B_{\t(r)}]_q$ similarly to the above. The result then follows by considering \cref{Eqn:ctr}.
    \end{proof}

This completes the proof that \cref{Eqn:ctrCheck1} is preserved under $\ct_r$.
We now show that \eqref{Eqn:ctrCheck2} is invariant under $\ct_r$. 

    \begin{proposition} \label{Prop:ctrCheck2}
        The relation
            \begin{equation}
                [ \ct_r(B_{r-1}), \ct_r(B_{\t(r-1)})] = \frac{1}{q-q^{-1}} \G_i
            \end{equation}
        holds in $\Beps$.
    \end{proposition}

    \begin{proof}
        Using the expressions for $\ct_r(B_{r-1})$ and $\ct_r(B_{\t(r-1)})$ given in \cref{Eqn:ctr(Br-1)_NEW,Eqn:ctr(Btr-1)_NEW} we have
            \begin{align*}
                \dfrac{1}{C^2}[\ct_r(B_{r-1}), \ct_r(B_{\t(r-1)})] 
                    &= \big[ [S,B_{\t(r)}]_q, [S^{\t}, B_r]_q \big] + \big[ \Delta, [S^{\t},B_r]_q \big]\\
                        &\quad{}+ \big[[S,B_{\t(r)}]_q, \Delta^{\t} \big] + [\Delta, \Delta^{\t}].
            \end{align*}
        By \cref{Lemma:Tech_Calc_1} we have
            \begin{align*}
                \big[ [S,B_{\t(r)}]_q, [S^{\t},B_r]_q \big] 
                    &= -\big[ \Delta, [S^{\t},B_r]_q \big] - \big[ [S,B_{\t(r)}]_q, \Delta^{\t} \big]\\
                        &\quad{}- \frac{q\cs_r \cs_{\t(r)}}{q-q^{-1}}(K_X - K_X^{-1})K_X\G_{r-1}.
            \end{align*}
        The result follows by recalling from \eqref{Eqn:ctr} that $C^2 = (q\cs_r \cs_{\t(r)})^{-1}$ and noting that
            \[ [\Delta, \Delta^{\t}] = \frac{q\cs_r\cs_{\t(r)}}{q-q^{-1}}K_X^2\G_{r-1}. \]
    \end{proof}

It remains to show that \cref{Eqn:ctrCheck3,Eqn:ctrCheck4} are invariant under $\ct_r$.

    \begin{proposition} \label{Prop:ctrCheck3}
        The relation
            \begin{equation}
                p(\ct_r(B_r), \ct_r(B_{r-1})) = 0
            \end{equation}
        holds in $\Beps$.
    \end{proposition}

    \begin{proof}
        Using the expression for $\ct_r(B_{r-1})$ from \cref{Eqn:ctr(Br-1)_NEW} and recalling that
            \[ \ct_r(B_r) = q^{-1}B_r\K{r}\Kw{r+1} \]
        we have
            \begin{align*}
                \frac{q^2}{C}p\big( \ct_r(B_r), &\ct_r(B_{r-1}) \big)\\
                    &= \big(B_r\K{r}\Kw{r+1} \big)^2\big( [S,B_{\t(r)}]_q + \Delta \big) + \big( [S,B_{\t(r)}]_q + \Delta \big)\big( B_r\K{r}\Kw{r+1} \big)^2\\
                    &\quad{}- (q+q^{-1})(B_r\K{r}\Kw{r+1})\big( [S,B_{\t(r)}]_q + \Delta \big)(B_r\K{r}\Kw{r+1}).
            \end{align*}
        By taking $(\K{r}\Kw{r+1})^2$ out as a factor, we obtain
            \begin{align} 
                \frac{q^2}{C}p\big( \ct_r&(B_r), \ct_r(B_{r-1}) \big)\big(\K{r}\Kw{r+1}\big)^{-2}\nonumber \\
                    &= B_r^2[S,B_{\t(r)}]_q - q^{-1}(q+q^{-1})B_r[S,B_{\t(r)}]_qB_r + q^{-2}[S,B_{\t(r)}]_qB_r^2 \label{Eqn:QSerre1_Proof1}\\
                        &\quad{} + B_r^2\Delta - q^{-1}(q+q^{-1})B_r\Delta B_r + q^{-2}\Delta B_r^2. \nonumber
            \end{align}
        By \cref{AppEqn:BrS} the element $B_r$ commutes with $S$ which implies
            \begin{equation} \label{Eqn:Br[SBtr]}
                B_r[S,B_{\t(r)}]_q = [S,B_{\t(r)}]_qB_r + \frac{1}{q-q^{-1}}[S,\G_r]_q.
            \end{equation}
        It follows from that
            \begin{align*}
                B_r^2[S,B_{\t(r)}]_q 
                    &= B_r[S,B_{\t(r)}]_qB_r + \frac{1}{q-q^{-1}}B_r[S,\G_r]_q, \\
                [S,B_{\t(r)}]_qB_r^2
                    &= B_r[S,B_{\t(r)}]_qB_r + \frac{1}{q-q^{-1}}[S,\G_r]_qB_r.
            \end{align*}
        Substituting these two expressions into \cref{Eqn:QSerre1_Proof1} we obtain
            \begin{align}
                \frac{q^2}{C}p\big( \ct_r(B_r), \ct_r(B_{r-1}) \big) \big(\K{r}\Kw{r+1}\big)^{-2}
                    &= \frac{1}{q-q^{-1}}\big[S, [B_r,\G_r]_{q^{-2}}\big]_q + B_r^2\Delta \nonumber \\ 
                        &\quad{}- q^{-1}(q+q^{-1})B_r\Delta B_r + q^{-2}\Delta B_r^2. \label{Eqn:QSerre1_Proof2}
            \end{align}
        Recall from \cref{AppLem:Ztr_S_comm} that
            \begin{equation} \label{Eqn:SZtr}
            S\Z_{\t(r)}= q\Z_{\t(r)}S - q(q-q^{-1})[B_{r-1}\K{r}K_X, B_r].
            \end{equation}
        Since 
            \[ [B_r, \G_r]_{q^{-2}} = -(q^2-q^{-2})\cs_{\t(r)}\Z_{\t(r)}B_r \]
        one calculates that
        \begin{align*}
            \frac{1}{q-q^{-1}}\big[S, [B_r,\G_r]_{q^{-2}}\big]_q
                &\overset{\phantom{\eqref{Eqn:SZtr}}}{=} -(q+q^{-1})\cs_{\t(r)}[S,\Z_{\t(r)}]_qB_r\\
                &\overset{\eqref{Eqn:SZtr}}{=}  (q^2-q^{-2})[\Delta,B_r]B_r.
        \end{align*}
        The result follows by substituting this expression into \cref{Eqn:QSerre1_Proof2} and observing that
            \[ [\Delta, B_r]B_r = q^{-2}B_r[\Delta,B_r] \]
        holds since $p(B_r, B_{r-1}) = 0$.
    \end{proof}

    \begin{proposition} \label{Prop:ctrCheck4}
        The relation
        \begin{equation*}
            p(\ct_r(B_{r-1}), \ct_r(B_r))=0
        \end{equation*}
        holds in $\Beps$.
    \end{proposition}

    \begin{proof}
        Using \cref{Eqn:ctr,Eqn:ctr(Br-1)_NEW} we have
        \begin{align}
            \frac{q}{C^2}p( \ct_r(B_{r-1}), &\ct_r(B_r) )\big(\K{r}\Kw{r+1}\big)^{-1}\nonumber\\
                &= \big( [S,B_{\t(r)}]_q + \Delta \big)^2B_r + q^2B_r\big( [S,B_{\t(r)}]_q + \Delta \big)^2 \nonumber\\
                &\quad{}- q(q+q^{-1})\big( [S,B_{\t(r)}]_q + \Delta \big)B_r\big( [S,B_{\t(r)}]_q + \Delta \big). \label{Eqn:QSerre2_Proof1}
        \end{align}
        Since $B_{r-1}[B_{r-1}, B_r]_q = q^{-1}[B_{r-1},B_r]_qB_{r-1}$ and $B_{r-1}$ commutes with $F_X^+$ and $B_{\t(r)}$ it follows that
            \begin{equation*}
                [S,B_{\t(r)}]_q\Delta = \Delta[S,B_{\t(r)}]_q.       
            \end{equation*}
        This implies
            \begin{equation*}
                \big([S,B_{\t(r)}]_q +\Delta \big)^2=[S,B_{\t(r)}]_q^2 + 2[S,B_{\t(r)}]_q\Delta + \Delta^2.
            \end{equation*}
        We consider terms involving different powers of $\Delta$ in \eqref{Eqn:QSerre2_Proof1} separately. First, we consider the expression
            \begin{equation} \label{Eqn:QSerre2_Proof2}
                [S,B_{\t(r)}]_q^2B_r - q(q+q^{-1})[S,B_{\t(r)}]_qB_r[S,B_{\t(r)}]_q + q^2B_r[S,B_{\t(r)}]_q^2.
            \end{equation}
        Using the relation 
            \begin{equation*}
                [S,B_{\t(r)}]_qB_r = B_r[S,B_{\t(r)}]_q - \frac{1}{q-q^{-1}}[S,\G_r]_q
            \end{equation*}
        from \cref{Eqn:Br[SBtr]} it follows that
            \begin{align} \label{Eqn:QSerre2_2}
                \begin{split}
                    [S,B_{\t(r)}]_q^2B_r - q(q+q^{-1})&[S,B_{\t(r)}]_qB_r[S,B_{\t(r)}]_q + q^2B_r[S,B_{\t(r)}]_q^2\\
                        &= \frac{q^2}{q-q^{-1}}[S,\G_r]_q[S,B_{\t(r)}]_q - \frac{1}{q-q^{-1}}[S,B_{\t(r)}]_q[S,\G_r]_q.
                \end{split}
            \end{align}
        By \cref{AppLem:Ztr_S_comm} we have
            \[ \cs_{\t(r)}[S,\Z_{\t(r)}]_q = -(q-q^{-1})[\Delta, B_r] \]
        which implies 
            \begin{align}
                [S,\G_r]_q[S,B_{\t(r)}]_q
                    &= \cs_r[S,\Z_r]_q[S,B_{\t(r)}]_q + (q-q^{-1})[\Delta,B_r][S,B_{\t(r)}]_q, \label{Eqn:QSerre2_3}\\
                [S,B_{\t(r)}]_q[S,\G_r]_q 
                    &= \cs_r[S,B_{\t(r)}]_q[S,\Z_r]_q + (q-q^{-1})[S,B_{\t(r)}]_q[\Delta,B_r]. \label{Eqn:QSerre2_4}
            \end{align}
        Using \cref{Lem:QSerre2_Rel1,Lemma:QSerre2_Rel2} we commute $[S,\Z_r]_q$ through $[S,B_{\t(r)}]_q$. In particular we have
            \begin{align*}
                \Z_r[S,B_{\t(r)}]_q 
                    &= q^{-1}[S,B_{\t(r)}]_q\Z_r + q^{-2}(q-q^{-1})\Delta\Z_r,\\
                S[S,B_{\t(r)}]_q 
                    &= q^{-1}[S,B_{\t(r)}]_qS -q^{-2}(q^2-q^{-2})S\Delta.
            \end{align*}
        Combining this with the relations 
            \begin{align*}
                S\Delta &= q^3\Delta S,\\
                \Z_r\Delta &= q^{-3}\Delta\Z_r
            \end{align*}
        we obtain 
             \begin{align}
                [S,\Z_r]_q[S,B_{\t(r)}]_q
                    &= q^{-2}[S,B_{\t(r)}]_q[S,\Z_r]_q - (1-q^{-2})\Delta[S,\Z_r]_q \label{Eqn:QSerre2_5}
            \end{align}
        Again by \cref{Eqn:Br[SBtr]} we have
         \begin{align}
                [\Delta, B_r][S,B_{\t(r)}]_q 
                    &= \Delta[S,B_{\t(r)}]_qB_r - B_r[S,B_{\t(r)}]_q\Delta + \frac{1}{q-q^{-1}}\Delta[S,\G_r]_q, \label{Eqn:QSerre2_6} \\ 
                [S,B_{\t(r)}]_q[\Delta, B_r]
                    &= -B_r[S,B_{\t(r)}]_q\Delta + \Delta[S,B_{\t(r)}]_qB_r + \frac{1}{q-q^{-1}}[S,\G_r]_q\Delta. \label{Eqn:QSerre2_7}
            \end{align}
        By Equation \eqref{Eqn:SZtr} we have
            \begin{equation} \label{Eqn:[S,Ztr]}
                \cs_{\t(r)}[S, \Z_{\t(r)}]_q = -(q-q^{-1})[\Delta, B_r].
            \end{equation}
        Substituting \cref{Eqn:QSerre2_5,Eqn:QSerre2_6} into \eqref{Eqn:QSerre2_3}, \cref{Eqn:QSerre2_7} into \eqref{Eqn:QSerre2_4} and using \eqref{Eqn:[S,Ztr]} we obtain
            \begin{align*}
                [S,\G_r]_q[S,B_{\t(r)}]_q
                    &= q^{-2}\cs_r[S,B_{\t(r)}]_q[S,\Z_r]_q + q^{-2}\cs_r\Delta[S,\Z_r]_q+ (q-q^{-1})\Delta[\Delta,B_r]\\
                        &\quad{}+ (q-q^{-1})\Delta[S,B_{\t(r)}]_qB_r - (q-q^{-1})B_r[S,B_{\t(r)}]_q\Delta,\\
                [S,B_{\t(r)}]_q[S,\G_r]_q
                    &= \cs_r[S,B_{\t(r)}]_q[S,\Z_r]_q + \cs_r[S,\Z_r]_q\Delta + (q-q^{-1})\Delta[S,B_{\t(r)}]_qB_r\\
                        &\quad{}-(q-q^{-1})B_r[S,B_{\t(r)}]_q\Delta + (q-q^{-1})[\Delta,B_r]\Delta.
            \end{align*}
         Hence \cref{Eqn:QSerre2_2} implies that
            \begin{align} \label{Eqn:QSerre2_8}
                \begin{split}
                [S,B_{\t(r)}]_q^2B_r - &q(q+q^{-1})[S,B_{\t(r)}]_qB_r[S,B_{\t(r)}]_q + q^2B_r[S,B_{\t(r)}]_q^2\\
                    &= (q^2-1)\Delta[S,B_{\t(r)}]_qB_r - (q^2-1)B_r[S,B_{\t(r)}]_q\Delta - \big[[\Delta,B_r],\Delta\big]_{q^2}.
                \end{split}
            \end{align}
        We next consider the expression 
            \[ 2\Delta[S,B_{\t(r)}]_qB_r - (1+q^2)\big([S,B_{\t(r)}]_qB_r\Delta + \Delta B_r[S,B_{\t(r)}]_q\big) + 2q^2B_r[S,B_{\t(r)}]_q\Delta. \]
        By \cref{Eqn:Br[SBtr],AppEqn:ZtrS} we have
            \begin{align*}
                [S,B_{\t(r)}]_qB_r\Delta
                    &= B_r[S,B_{\t(r)}]_q\Delta - \frac{1}{q-q^{-1}}\cs_r[S,\Z_r]_q\Delta - [\Delta,B_r]\Delta,\\
                \Delta B_r[S,B_{\t(r)}]_q 
                    &= \Delta[S,B_{\t(r)}]_qB_r + \frac{1}{q-q^{-1}}\cs_r\Delta[S,\Z_r]_q + \Delta[\Delta,B_r]
            \end{align*}
        from which it follows that
            \begin{align} \label{Eqn:QSerre2_9}
                \begin{split}
                    2\Delta[S,&B_{\t(r)}]_qB_r - (1+q^2)\big([S,B_{\t(r)}]_qB_r\Delta + \Delta B_r[S,B_{\t(r)}]_q\big) + 2q^2B_r[S,B_{\t(r)}]_q\Delta\\
                        &= (q^2-1)B_r[S,B_{\t(r)}]_q\Delta - (q^2-1)\Delta[S,B_{\t(r)}]_qB_r + (q^2+1)\big[[\Delta,B_r],\Delta\big].
                \end{split}
            \end{align}
        Combining \cref{Eqn:QSerre2_8,Eqn:QSerre2_9} and substituting into \cref{Eqn:QSerre2_Proof1} gives
            \begin{align*}
                \frac{q}{C^2}p&\big(\ct_r(B_{r-1}), \ct_r(B_r) \big) \big(\K{r}\Kw{r+1}^{-1}\big)\\
                    &= -\big[[\Delta,B_r],\Delta\big]_{q^2} + (q^2+1)\big[[\Delta,B_r],\Delta\big] + \Delta^2B_r - q(q+q^{-1})\Delta B_r\Delta\\
                        &\quad{}+ q^2B_r\Delta^2\\
                    &= 0
            \end{align*}
        as required.
    \end{proof}

\Cref{Lemma:ctrCheck1-Part1,Lemma:ctrCheck1-Part2,Lemma:ctrCheck1-Part3,Prop:ctrCheck2,Prop:ctrCheck3,Prop:ctrCheck4} together show that $\ct_r$ is an algebra endomorphism of $\Beps$. 


\subsection{Proof that $\ct_r^{-1}$ is an algebra endomorphism}
We now show that $\ct_r^{-1}$ given in Equation \eqref{Eqn:ctri} also defines an algebra endomorphism of $\Beps$. First, with an additional constraint on the parameters $\be$, we construct an anti-involution $\phi:\Beps \rightarrow \Beps$ such that
\begin{equation}
	\ct_r^{-1} = \phi \circ \ct_r \circ \phi
\end{equation}
holds, compare with \cite[Section~37.2.4]{b-Lus94}. In particular, set
\begin{align} \label{Eqn:anti-inv}
	\begin{split}
		\phi(B_i) = B_i, \quad \phi(L_i) = L_{\t(i)} \quad \mbox{for $i \in I \setminus X$,}\\
		\phi(E_j) = E_j, \quad \phi(F_j) = F_j, \quad \phi(K_j) = K_j^{-1} \quad \mbox{for $i \in X$.}
	\end{split}
\end{align}
Note that the restriction of $\phi$ on $\mathcal{M}_X$ coincides with Lusztig's algebra anti-automorphism $\t$ on $\Uqg$, see \cite[Section~3.1.3]{b-Lus94}.

\begin{lemma} \label{lemma:phi_anti-inv}
The map $\phi$ given by Equation \eqref{Eqn:anti-inv} is an involutive algebra anti-automorphism of $\Beps$ if and only if $\cs_r = \cs_{\t(r)}$.
\end{lemma}

\begin{proof}
By considering the defining relations of $\Beps$ given in Equations \eqref{Eqn:BiRel1} -- \eqref{Eqn:BiRel4}, it suffices to show that $\phi$ preserves the relation 
\begin{equation} \label{Eqn:anti-inv_check}
	B_iB_{\t(i)} - B_{\t(i)}B_i = \frac{1}{q-q^{-1}}(\cs_i\Z_i - \cs_{\t(i)}\Z_{\t(i)})
\end{equation}
for $i \in I \setminus X$. Since $\phi(E_X^+) = E_X^-$ and $\phi(E_X^-) = E_X^+$, it follows that
\begin{equation*}
	\phi(\Z_i) = \Z_{\t(i)} \quad \mbox{for all $i \in I \setminus X$.}
\end{equation*}
Since
\begin{equation*}
	\phi(B_iB_{\t(i)} - B_{\t(i)}B_i) = B_{\t(i)}B_i - B_iB_{\t(i)} = \frac{1}{q-q^{-1}}(\cs_{\t(i)}\Z_{\t(i)} - \cs_i\Z_i)
\end{equation*}
holds for all $i \in I \setminus X$, Equation \eqref{Eqn:anti-inv_check} is preserved by $\phi$ if and only if
\begin{equation*}
	\cs_{\t(i)}\Z_{\t(i)} - \cs_i\Z_i = \cs_i\Z_{\t(i)} - \cs_{\t(i)}\Z_i
\end{equation*} 
for all $i \in I \setminus X$. This holds if and only if $\cs_i = \cs_{\t(i)}$ for all $i \in I \setminus X$. The result follows from this and Equation \eqref{Eqn:cs_cond}.
\end{proof}

By the above lemma and the results of the previous section, the map $\phi \circ \ct_r \circ \phi$ is an algebra endomorphism of $\Beps$ if $\cs_r = \cs_{\t(r)}$. 

\begin{lemma} \label{lemma:ctri=pctip}
	Suppose $\cs_r = \cs_{\t(r)}$. Then for $i \in I \setminus X$ the relation 
	\begin{equation*}
	\ct_r^{-1}(B_i) = \phi \circ \ct_r \circ \phi(B_i)
	\end{equation*}
	holds.
\end{lemma}

\begin{proof}
	By Equation \eqref{Eqn:ctri} it suffices to only consider $i = r-1$ and $i = r$. 
	The result for $i = r$ follows since
		\begin{equation*}
			\phi(B_rL_r\Kw{r+1}) = q^2B_rL_{\t(r)}\Kw{\t(r+1)}
		\end{equation*}
	and hence
	\begin{equation*}
		\phi \circ \ct_r \circ \phi(B_r) = qB_rL_{\t(r)}\Kw{\t(r+1)} = \ct_r^{-1}(B_{r-1}).
	\end{equation*}
	Recall from Equation \eqref{Eqn:ctr(Br-1)_NEW} that we have
	\begin{equation*}
		\ct_r(B_{r-1}) = C\big(\big[[[B_{r-1}, B_r]_q, F_X^+]_q, B_{\t(r)} \big]_q + q\cs_{\t(r)}B_{r-1}L_rK_X \big).
	\end{equation*}
	Since $\phi$ is an anti-involution of $\Beps$ we have
	\begin{align*}
		\phi\big( \big[[[B_{r-1}, B_r]_q, F_X^+]_q, B_{\t(r)} \big]_q \big) &= \big[B_{\t(r)}, [F_X^-, [B_r, B_{r-1}]_q]_q \big]_q,\\
		\phi(B_{r-1}L_rK_X) &= q^{-1}B_{r-1}L_{\t(r)}K_X^{-1}.
	\end{align*}
	Combining both relations and comparing with Equation \eqref{Eqn:ctri} gives
	\begin{align*}
		\phi \circ \ct_r \circ \phi(B_{r-1}) 
			&= C\big(\big[B_{\t(r)}, [F_X^-, [B_r, B_{r-1}]_q]_q \big]_q + \cs_rB_{r-1}L_{\t(r)}K_X^{-1} \big)\\
			&= \ct_r^{-1}(B_{r-1})
	\end{align*}
	as required, where we have additionally used the condition $\cs_r = \cs_{\t(r)}$. 
\end{proof}

As a consequence of Lemma \ref{lemma:ctri=pctip} and Section \ref{Subsec:ctr_endo} the following corollary is immediate.

\begin{corollary} \label{Cor:ctri_endo_withcond}
	If $\cs_r = \cs_{\t(r)}$ then $\ct_r^{-1}$ given by Equation \eqref{Eqn:ctri} defines an algebra endomorphism of $\Beps$.
\end{corollary}

We now show that $\ct_r^{-1}$ defines an algebra endomorphism of $B_{\be}$ also for general parameters. When necessary, we specify the dependence on the parameters $\be$ explicitly by writing $\ct_{r,\be}$ for $\ct_r$ and $B_i^{\be}$ for $B_i$, where $i \in I \setminus X$. Additionally, we also write $C^{\be}$ for the constant $C$ appearing in $\ct_r(B_{r-1})$ and $\ct_r(B_{\t(r-1)})$. 
By \cite[Lemma~2.5.1]{a-Wat19}, there exists a field extension $\field_1$ of $\field(q^{1/2})$ (\cite[Section~2.4]{a-Wat19}) and additional parameters $\bet, \bz \in (\field_1^\times)^I$ satisfying
\begin{align}
	\zeta_i\zeta_{\t(i)} &= 1 \quad \mbox{for all $i \in I$,} \label{Eqn:Wat_cond1} \\
	\zeta_i &= 1 \quad \mbox{unless $i = r, \t(r)$,} \label{Eqn:Wat_cond2} \\
	\cs_i\eta_i\eta_{\t(i)}\zeta_i &= \cs'_i \quad \mbox{for all $i \in I$} \label{Eqn:Wat_cond3} 
\end{align}
such that the map $\mathcal{A}_{\bet,\bz}:B_{\be}\rightarrow B_{\be'}$ with $\mathcal{A}_{\bet,\bz}|_{\mathcal{M}_X} = \text{id}|_{\mathcal{M}_X}$ and
\begin{align}
	\mathcal{A}_{\bet,\bz}(B_i^{\be}) &= \eta_i^{-1}B_i^{\be'}, \label{Eqn:A(B)} \\
	\mathcal{A}_{\bet,\bz}(L_i) &= \zeta_i^{-1}L_i, \label{Eqn:A(L)} \\
	\mathcal{A}_{\bet,\bz}(K_j) &= K_j
\end{align}
for all $i \in I \setminus X$ and $j \in X$ is an algebra isomorphism. Here, we consider the algebras $B_{\be}$ and $B_{\be'}$ over $\field_1$. We note that $\Aez$ restricted to the semisimple part is the identity. In particular, we have $\Aez(K_{\varpi_j}) = K_{\varpi_j}$ for $j \in X$.

 Additionally, by \eqref{Eqn:BiRel3} and \eqref{Eqn:Wat_cond1} the map $f:B_{\be'} \rightarrow B_{\be'}$ given by
\begin{equation} \label{Eqn:f_map}
	f(B_i) = 	\begin{cases}
					\zeta_i^{-1}B_i & \mbox{if $i = r, \t(r)$,}\\
					B_i & \mbox{otherwise}
				\end{cases}
\end{equation}
and $f|_{\mathcal{M}_XU_0^{\Theta}} = \text{id}|_{\mathcal{M}_XU_0^{\Theta}}$ is an algebra automorphism of $B_{\be}$.

\begin{lemma} \label{Lemma:ctri_alg_end}
	For any $\be$ and $\be'$ satisfying \eqref{Eqn:cs_cond} and for any $\bet, \bz \in (\field_1^{\times})^I$ satisfying \eqref{Eqn:Wat_cond1}, \eqref{Eqn:Wat_cond2} and \eqref{Eqn:Wat_cond3}, the relation
	\begin{equation} \label{Eqn:ctri_gen}
		\ct_{r,\be'}^{-1}(B_i^{\be'}) = \mathcal{A}_{\bet,\bz} \circ \ct_{r,\be}^{-1} \circ \mathcal{A}_{\bet,\bz}^{-1} \circ f(B_i^{\be'})
	\end{equation}
	holds for any $i \in I \setminus X$.
\end{lemma}

\begin{proof}
By Equations \eqref{Eqn:ctri}, \eqref{Eqn:A(B)} and \eqref{Eqn:f_map}, it suffices to verify \eqref{Eqn:ctri_gen} on the elements $B_r$ and $B_{r-1}$. We have
\begin{align*}
	\Aez \circ \ct_{r,\be}^{-1} \circ \Aez^{-1} \circ f(B_r^{\be'})
		&\overset{\phantom{\eqref{Eqn:ctri}}}{=}	\zeta_r^{-1}\eta_r \Aez \circ \ct_{r,\be}^{-1}(B_r^{\be}) \\
		&\overset{\eqref{Eqn:ctri}}{=}	q\zeta_r^{-1}\eta_r \Aez(B_r^{\be}L_{\t(r)}\Kw{\t(r+1)}) \\
		&\overset{\phantom{\eqref{Eqn:ctri}}}{=} q\zeta_r^{-1}\zeta_{\t(r)}^{-1}B_r^{\be'}L_{\t(r)}\Kw{\t(r+1)} \\
		&\overset{\eqref{Eqn:Wat_cond1}}{=} \ct_{r,\be'}^{-1}(B_r^{\be'})
\end{align*}
and hence \eqref{Eqn:ctri_gen} holds on $B_r^{\be'}$.

Using relations \eqref{Eqn:Wat_cond1} and \eqref{Eqn:Wat_cond3} we obtain
\begin{align*}
C^{\be}\eta_r^{-1}\eta_{\t(r)}^{-1} 
	&=	(q\cs_r\cs_{\t(r)})^{-1/2} \eta_r^{-1}\eta_{\t(r)}^{-1} \\
	&=	(q\eta_r^{-2}\eta_{\t(r)}^{-2}\zeta_r^{-1}\zeta_{\t(r)}^{-1}\cs_r'\cs_{\t(r)}')^{-1/2} \eta_r^{-1}\eta_{\t(r)}^{-1} \\
	&= C^{\be'}.
\end{align*}
Hence
\begin{align*}
	\Aez \circ \ct_{r,\be}^{-1} &\circ \Aez^{-1} \circ f(B_{r-1}^{\be'})\\
		&=	\eta_{r-1}\Aez \circ \ct_{r,\be}^{-1}(B_{r-1}^{\be})\\
		&=	\eta_{r-1}C^{\be}\Aez\big(\big[B_{\t(r)}^{\be}, [F_X^{-}, [B_r^{\be}, B_{r-1}^{\be}]_q ]_q \big]_q + \cs_rB_{r-1}^{\be}L_{\t(r)}K_X^{-1} \big)\\
		&=  C^{\be'}\big[B_{\t(r)}^{\be'}, [F_X^{-}, [B_r^{\be'}, B_{r-1}^{\be'}]_q ]_q \big]_q + C^{\be'}\cs_r\eta_{r}\eta_{\t(r)}\zeta_{\t(r)}^{-1}B_{r-1}^{\be'}L_{\t(r)}K_X^{-1}\\
		&=	\ct_{r,\be'}^{-1}(B_{r-1}^{\be'})
\end{align*}
as required, where the last equality follows again from relations \eqref{Eqn:Wat_cond1} and \eqref{Eqn:Wat_cond3}. 
\end{proof}

Recall by Corollary \ref{Cor:ctri_endo_withcond} that $\ct_{r,\be}^{-1}$ is an algebra endomorphism of $B_{\be}$ if $\cs_r = \cs_{\t(r)}$. Since $\Aez$ and $f$ are algebra homomorphisms, Lemma \ref{Lemma:ctri_alg_end} implies the following corollary.

\begin{corollary} \label{Cor:ctri_endo_general}
	For any $\be$ satisfying \eqref{Eqn:cs_cond}, the map $\ct_{r,\be}^{-1}$ is an algebra endomorphism of $B_{\be}$.
\end{corollary}

\begin{remark}
	It is possible to complete similar calculations as in Section \ref{Subsec:ctr_endo} in order to show that $\ct_r^{-1}$ is an algebra endomorphism of $B_{\be}$. This way, it is not necessary to consider the field extension $\field_1^{\times}$ of $\field^{\times}$, which is required in the construction of the algebra automorphism $\Aez:B_{\be} \rightarrow B_{\be'}$.
\end{remark}

\subsection{Proof that $\ct_r^{-1}$ is the inverse of $\ct_r$}
 In order to complete the proof of \cref{Thm:ctr_Aut}, we now show that $\ct_r \circ \ct_r^{-1} = \ct_r^{-1} \circ \ct_r = \text{id}$. It suffices to check this on the generators $B_i$ for $i \in I \setminus X$. The relation
    \[ \ct_r \circ \ct_r^{-1}(B_i) = B_i = \ct_r^{-1} \circ \ct_r(B_i) \]
is straightforward for all $i \in I \setminus X$ except $i = r-1$ and $i = \t(r-1)$. Here we only consider the case $i = r-1$ since the case $i = \t(r-1)$ is analogous.

    \begin{proposition} \label{Prop:ctrinv_Part1}
        The relation
            \begin{equation}
                \ct_r \circ \ct_r^{-1}(B_{r-1}) = B_{r-1}
            \end{equation}
        holds in $\Beps$.
    \end{proposition}

    \begin{proof}
    By \cref{Eqn:ctr,Eqn:ctri} we have
        \begin{align}
            \begin{split}
                \frac{1}{C^2} \ct_r \circ \ct_r^{-1}(B_{r-1})
                    &= \big[q^{-1}B_{\t(r)}\K{\t(r)}, [F_X^-, [q^{-1}B_r\K{r}, [S,B_{\t(r)}]_q + \Delta]_q]_q \big]_q\\
                    &\quad{}+ \cs_r([S,B_{\t(r)}]_q + \Delta)\K{\t(r)}K_X^{-1}. \label{Proof:ctrinv_Part1}
            \end{split}
        \end{align}
    We consider the first summand of the above expression and simplify it. By \cref{AppEqn:BrS}, the element $B_r$ commutes with $S$. It follows that
        \begin{align*}
            \big[ q^{-1}B_r\K{r}, [S,B_{\t(r)}]_q + \Delta\big]_q
                &= \big[B_r, [S,B_{\t(r)}]_q\big]\K{r} + [B_r,\Delta]\K{r}\\
                &= \big[S, [B_r,B_{\t(r)}] \big]_q\K{r} + [B_r,\Delta]\K{r}\\
                &= \frac{1}{q-q^{-1}}[S,\G_r]_q\K{r} + [B_r,\Delta]\K{r}.
        \end{align*}
    We now commute this with $F_X^-$. Using \cref{AppEqn:ZtrS} and \cref{AppEqn:ctrinv1} we obtain
        \begin{align*}
            \big[F_X^-, [S,\G_r]_q\big]_q 
                &= \cs_r\big[F_X^-, [S,\Z_r]_q\big]_q - \cs_{\t(r)}\big[F_X^-, [S,\Z_{\t(r)}]_q\big]_q\\
                &= q(q-q^{-1})\cs_rS\K{\t(r)}K_X^{-1} + (q-q^{-1})\big[F_X^-,[\Delta,B_r]\big]_q.
        \end{align*}
    It follows from that
        \begin{align*}
            \big[ F_X^-, [q^{-1}B_r\K{r}, [S,B_{\t(r)}]_q+ \Delta]_q\big]_q 
                &= \frac{1}{q-q^{-1}}\big[F_X^-, [S,\G_r]_q\big]_q\K{r} + \big[F_X^-, [B_r,\Delta]\big]\K{r}\\
                &= q\cs_rSK_X^{-1}
        \end{align*}
    We now $q$-commute $q^{-1}B_{\t(r)}\K{\t(r)}$ and $q\cs_rSK_X^{-1}$ which gives
        \begin{align*}
            \big[q^{-1}B_{\t(r)}\K{\t(r)}, q\cs_rSK_X^{-1}\big]_q
                &= q\cs_r[B_{\t(r)},S]_{q^{-1}}\K{\t(r)}K_X^{-1}\\
                &= -\cs_r[S,B_{\t(r)}]_q\K{\t(r)}K_X^{-1}.
        \end{align*}
    Substituting the above into \cref{Proof:ctrinv_Part1} we obtain
        \begin{align*}
            \frac{1}{C^2} \ct_r \circ \ct_r^{-1}(B_{r-1})
                &= \cs_r\Delta \K{\t(r)}K_X^{-1}\\
                &= q\cs_r\cs_{\t(r)}B_{r-1}\\
                &= \frac{1}{C^2} B_{r-1}
        \end{align*}
    and hence we have $\ct_r \circ \ct_r^{-1}(B_{r-1}) = B_{r-1}$ as required.
    \end{proof}

In the proof of the following proposition, we write
\begin{equation}
\ct_r^{-1}(B_{r-1}) = C\big([B_{\t(r)}, T]_q + \Lambda\big)
\end{equation}
where
\begin{align}
T &= \big[F_X^-, [B_r, B_{r-1}]_q \big]_q, \\
\Lambda &= \cs_rB_{r-1}\K{\t(r)}K_X^{-1}.
\end{align}
    \begin{proposition} \label{Prop:ctrinv_Part2}
        The relation
            \begin{equation}
                \ct_r^{-1} \circ \ct_r(B_{r-1}) = B_{r-1}
            \end{equation}
        holds in $\Beps$.
    \end{proposition}

    \begin{proof}
    By \cref{Eqn:ctri,Eqn:ctr(Br-1)_NEW} we have
        \begin{align}
            \begin{split}
                \frac{1}{C^2} \ct_r^{-1} \circ \ct_r(B_{r-1})
                    &= \big[ [ [ [B_{\t(r)}, T]_q + \Lambda, qB_r\K{\t(r)}]_q, F_X^+]_q, qB_{\t(r)}\K{r} \big]_q\\
                    &\quad{} q\cs_{\t(r)}([B_{\t(r)},T]_q + \Lambda)\K{r}K_X. \label{Proof:ctrinv_Part2}
            \end{split}
        \end{align}
    By a similar proof to \cref{AppLem:ST_comms} the elements $B_r$ and $T$ commute. It follows that
        \begin{align}
            \big[[B_{\t(r)},T]_q + \Lambda, qB_r\K{\t(r)} \big]_q
                &= q\big[[B_{\t(r)},T]_q,B_r\big]\K{\t(r)} + q[\Lambda,B_r]\K{\t(r)} \nonumber \\
                &= q\big[[B_{\t(r)},B_r]_q, T \big]\K{\t(r)} + q[\Lambda,B_r]\K{\t(r)} \nonumber \\
                &= -\frac{q}{q-q^{-1}}[\G_r,T]_q\K{\t(r)} + q[\Lambda,B_r]\K{\t(r)}. \label{Proof:ctrinv_Part2-2}
        \end{align}
    Similarly to Equation \eqref{AppEqn:ZtrS} we have
    \begin{equation}
    	\Z_r T 
    		= q T \Z_r - (q-q^{-1}) [B_r,B_{r-1}]_q \K{\t(r)}K_X^{-1}. \label{AppEqn:ZrT}
    \end{equation}
    It follows that
        \begin{align*}
            -\frac{q}{q-q^{-1}}\cs_r\big[[\Z_r,T]_q,F_X^+\big]_q \K{\t(r)}
                &= q\cs_r\big[[B_r,B_{r-1}]_q\K{\t(r)}K_X^{-1}, F_X^+ \big]_q\K{\t(r)}\\
                &= q\big[ [B_r, \cs_rB_{r-1}\K{\t(r)}K_X^{-1}], F_X^+ \big]_q\K{\t(r)}\\
                &= -q\big[ [\Lambda,B_r],F_X^+\big]_q\K{\t(r)}.
        \end{align*}
    As in the proof of Lemma \ref{AppLem:ctrinv} the relation
    \begin{equation}
    \big[ [\Z_{\t(r)}, T]_q, F_X^+ \big]_q = (1-q^{-2})T\K{r}K_X \label{AppEqn:ctrinv2}
    \end{equation} 
    holds in $\Beps$. This and \cref{Proof:ctrinv_Part2-2} implies that
        \begin{align*}
            \big[[ [B_{\t(r)}&, T]_q + \Lambda, qB_r\K{\t(r)}]_q, F_X^+ \big]_q\\
                &= \frac{q}{q-q^{-1}}\cs_{\t(r)}\big[[\Z_{\t(r)}, T]_q,F_X^+ \big]_q\K{\t(r)} - \frac{q}{q-q^{-1}}\cs_r\big[[\Z_r,T]_q, F_X^+\big]_q\K{\t(r)}\\
                &\quad{}+ q\big[ [\Lambda, B_r], F_X^+\big]_q\K{\t(r)}\\
                &= \cs_{\t(r)}T K_X.
        \end{align*}
    We substitute this into \cref{Proof:ctrinv_Part2} to obtain
        \begin{align*}
            \frac{1}{C^2} \ct_r^{-1} \circ \ct_r(B_{r-1}) 
                &= [ \cs_{\t(r)}TK_X, qB_{\t(r)}\K{r}]_q + q\cs_{\t(r)}([B_{\t(r)},T]_q + \Lambda)\K{r}K_X\\
                &= q^2\cs_{\t(r)}[T,B_{\t(r)}]_{q^{-1}}\K{r}K_X + q\cs_{\t(r)}([B_{\t(r)},T]_q + \Lambda)\K{r}K_X\\
                &= q\cs_{\t(r)}\Lambda \K{r}K_X\\
                &= \frac{1}{C^2}B_{r-1}
        \end{align*}
    as required.
    \end{proof}


\section{Proof of \texorpdfstring{\cref{Thm:ctr_braid}}{Theorem2}} \label{Sec:Proof2}
Restricted to $\mathcal{M}_X\Uow$ the automorphisms $\ct_i$ act as the Lusztig automorphism $\widet{T}_i = T_{\widet{\sigma}_i}$ for $i \in I \setminus X$. As a result, the braid relations of \cref{Thm:ctr_braid} hold on elements of $\mathcal{M}_X\Uow$. Hence it suffices to verify \cref{Thm:ctr_braid} on the elements $B_i$ for $i \in I \setminus X$.

\subsection{Braid relations I}    
We first check that the relation $\ct_r\ct_i = \ct_i\ct_r$ holds for all $1 \leq i \leq r-2$.

    \begin{proposition}
        For $1 \leq i \leq r-2$ and $j \in I \setminus X$ the relation
            \begin{equation} \label{Eqn:ctrBraid1}
                \ct_r\ct_i(B_j) = \ct_i\ct_r(B_j)
            \end{equation}
        holds.
    \end{proposition}
    
    \begin{proof}
        By symmetry, we only check \cref{Eqn:ctrBraid1} for $1 \leq j \leq r$. This is done by a case-by-case analysis.
       \begin{case}$a_{ij}=0, a_{jr}=2$.
       \newline
       In this case we have $j = r$ and hence $\ct_i(B_j) = B_j$. This implies 
        \[ \ct_r\ct_i(B_j) = \ct_r(B_j) = q^{-1}B_rK_rK_{\t(r)}^{-1}K_{\varpi'_{r+1}} =\ct_i\ct_r(B_j) \]
        as required.
       \end{case}
       
       \begin{case}$a_{ij}=0, a_{jr}= -1$.
       	\newline
            Then $j = r-1$ and $\ct_i(B_j) = B_j$ hence
                \[\ct_r\ct_i(B_j) = \ct_r(B_j) = \ct_i\ct_r(B_j) \]
            as required.
       \end{case}
       
       \begin{case} $a_{ij} = 0, a_{jr} = 0$.
       \newline
        In this case, we have $\ct_r(B_j) = B_j$ and $\ct_i(B_j) = B_j$ so the statement of the proposition holds.
       \end{case}
       
       \begin{case} $a_{ij}= -1, a_{jr} = 0$.
       \newline
        Here, we have $\ct_r(B_j) = B_j$ and $\ct_i(B_j) = (q\cs_i)^{-1/2}[B_j,B_i]_q$. Hence
            \begin{align*}
                \ct_r\ct_i(B_j) 
                    &=(q\cs_i)^{-1/2}[\ct_r(B_j), \ct_r(B_i)]_q\\
                    &=(q\cs_i)^{-1/2}[B_j,B_i]_q = \ct_i\ct_r(B_j).
            \end{align*}
       \end{case}
       
       \begin{case} $a_{ij} =-1, a_{jr} =-1$. 
       \newline
        This case can only occur if $i = r-2$ and $j = r-1$. Then by \cref{Eqn:ctr} we have
        \begin{align*}
            \ct_{r-2}\ct_r(B_{r-1})
                &= C\ct_{r-2}\Big( \big[B_{r-1},[B_r, [F_X^+,B_{\t(r)}]_q]_Q\big]_q + q\cs_{\t(r)}B_{r-1}K_rK_{\t(r)}^{-1}K_X^{-1}\Big)\\
                &= q^{-1/2}C \Big(\big[[B_{r-1},B_{r-2}]_q,[B_r, [F_X^+,B_{\t(r)}]_q]_q\big]_q\\ 
                    &\quad{}+ q\cs_{\t(r)}[B_{r-1},B_{r-2}]_qK_rK_{\t(r)}^{-1}K_X^{-1}\ \Big)\\
                &=q^{-1/2}C \Big(\big[ [B_{r-1},[B_r, [F_X^+,B_{\t(r)}]_q]_q]_q. B_{r-2}\big]_q \\
                    &\quad{}+ q\cs_{\t(r)}[B_{r-1}K_rK_{\t(r)}^{-1}K_X^{-1},B_{r-2}]_q \Big)\\
                &=q^{-1/2}[\ct_r(B_{r-1}),\ct_r(B_{r-2})]_q\\
                &=\ct_r\ct_{r-2}(B_{r-1})
        \end{align*}
        as required.
       \end{case}
       
       \begin{case} $a_{ij}= 2$
       \newline
       Then we have $\ct_r(B_i) = B_i$ and $\ct_r(\ct_i(B_i)) = \ct_i(B_i)$ which implies the result in this case. This completes the proof.
       \end{case}
    \end{proof}


\subsection{Braid relations II}    
We now check that the relation 
    \begin{equation} \label{Eqn:BtypeBraid}
     \ct_r\ct_{r-1}\ct_r\ct_{r-1}(B_j) = \ct_{r-1}\ct_r\ct_{r-1}\ct_r(B_j)
     \end{equation}
holds for all $j \in I \setminus X$. Again for symmetry reasons it is enough to only consider $1 \leq j \leq r$. Many of the remaining proofs in this section require the use of relations that are proven in \cref{Appendix}. Since $\ct_r(B_j)= B_j$ and $\ct_{r-1}(B_j)= B_j$ for $1 \leq j < r-2$ the following lemma is immediate.

\begin{lemma}
    For $1 \leq j < r-2$ the relation \eqref{Eqn:BtypeBraid} holds.
\end{lemma}

As a result of the above lemma, it remains to verify \cref{Eqn:BtypeBraid} for $j \in \{ r-2,r-1,r\}$. For the next result, we use the relation
    \begin{equation} \label{Eqn:Rel}
     \ct_{r-1}\ct_r(B_{r-1}) = \ct_r^{-1}(B_{\t(r-1)})
     \end{equation}
which appears in the proof of \cref{AppLem:3}.

\begin{proposition} \label{Prop:Braid-j=r-1}
   For $j = r-1$ the relation \eqref{Eqn:BtypeBraid} holds.
\end{proposition}

\begin{proof}
    Using \cref{Eqn:Rel} we have
        \begin{align*}
            \ct_r\ct_{r-1}\ct_r\ct_{r-1}(B_{r-1})
                &= \ct_r\ct_{r-1}\ct_r\big(q^{-1}B_{\t(r-1)}K_{\t(r-1)}K_{r-1}^{-1} \big)\\
                &= q^{-1}B_{r-1}K_{r-1}K_{\t(r-1)}^{-1}\\
                &= \ct_{r-1}(B_{\t(r-1)}).
        \end{align*}
    The result follows from Lemma \ref{AppLem:3}.
\end{proof}

\begin{proposition}
    For $j = r-2$ the relation \eqref{Eqn:BtypeBraid} holds. 
\end{proposition}

\begin{proof}
    On one hand we have
    \begin{align*}
        \ct_r\ct_{r-1}\ct_r\ct_{r-1}(B_{r-2})
            &= q^{-1/2}\ct_r\ct_{r-1}\ct_r([B_{r-2},B_{r-1}]_q)\\
            &= q^{-1/2}\ct_r\ct_{r-1}([B_{r-2}, \ct_r(B_{r-1})]_q)\\
            &= q^{-1}\ct_r\big( \big[ [B_{r-2}, B_{r-1}]_q, \ct_{r-1}\ct_r(B_{r-1}) \big]_q \big)
    \end{align*}
    Again by \cref{Eqn:Rel} it follows that
    \begin{align*}
        \ct_r\ct_{r-1}\ct_r\ct_{r-1}(B_{r-2})
            &= q^{-1}\big[[B_{r-2}, \ct_r(B_{r-1})]_q, B_{\t(r-1)} \big]_q\\
            &= q^{-1}\big[ B_{r-2}, [\ct_r(B_{r-1}), B_{\t(r-1)}]_q\big]_q
    \end{align*}
    where the last equality follows since $B_{r-2}$ commutes with $B_{\t(r-1)}$. 
    On the other hand we have
    \begin{align*}
        \ct_{r-1}\ct_r\ct_{r-1}\ct_r(B_{r-2})
            &= \ct_{r-1}\ct_r\ct_{r-1}(B_{r-2})\\
            &= q^{-1}\big[ [B_{r-2}, B_{r-1}]_q, \ct_r^{-1}(B_{\t(r-1)})\big]_q.
    \end{align*}
    Since $B_{\t(r-1)}$ commutes with $\ct_r(B_{r-2})$ it follows that
     \begin{equation} \label{Eqn:B(r-2)comm}
     B_{r-2}\ct_r^{-1}(B_{\t(r-1)}) = \ct_r^{-1}(B_{\t(r-1)})B_{r-2}.
    \end{equation}
    By \cref{AppCor:5} the element $[B_{r-1},\ct_r^{-1}(B_{\t(r-1)})]_q$ is invariant under $\ct_r$. This and \cref{Eqn:B(r-2)comm} imply
    \begin{align*}
        \ct_{r-1}\ct_r\ct_{r-1}\ct_r(B_{r-2})
            &= q^{-1}\big[ B_{r-2}, [B_{r-1},\ct_r^{-1}(B_{\t(r-1)})]_q \big]_q\\
            &= q^{-1}\big[ B_{r-2},[\ct_r(B_{r-1}), B_{\t(r-1)}]_q \big]_q\\
            &= \ct_r\ct_{r-1}\ct_r\ct_{r-1}(B_{r-2})
    \end{align*}
    as required.
\end{proof}

\begin{proposition}
   For $j = r$ the relation \eqref{Eqn:BtypeBraid} holds.
\end{proposition}
    
\begin{proof}
Consider first the term $\ct_{r-1}\ct_r\ct_{r-1}\ct_r(B_r)$. By \cref{Eqn:cti,Eqn:ctr}, \cref{Eqn:Rel} and $\ct_i|_{\mathcal{M}_X\Uow} = \widet{T}_i|_{\mathcal{M}_X\Uow}$ for $1 \leq i \leq r$ we obtain
    \begin{align*}
        \ct_{r-1}&\ct_r\ct_{r-1}\ct_r(B_r)\\
            &= q^{-3}\big[ [B_r,B_{r-1}]_q\K{r}\K{r-1}K_{\varpi'_{r+1}}, \ct_r^{-1}(B_{\t(r-1)}) \big]_q \K{r}K_{\varpi'_{r+1}}\\
            &= q^{-2}\big[[B_r,B_{r-1}]_q,\ct_r^{-1}(B_{\t(r-1)}) \big]_q \K{r}^2\K{r-1}K_{\varpi'_{r+1}}^2 
    \end{align*}        
    where the second equality follows from \cref{Eqn:BiRel1} and noting that $K_{\varpi'_{r+1}}$ commutes with $\ct_r^{-1}(B_{\t(r-1)})$. Since $[\ct_r(B_r), B_{\t(r-1)}]_{q^{-1}} = 0$, it follows that $[B_r, \ct_r^{-1}(B_{\t(r-1)})]_{q^{-1}} = 0$. Using this and \cref{AppCor:5} it follows that
    \begin{align*}
        \big[ [B_r,B_{r-1}]_q, \ct_r^{-1}(B_{\t(r-1)}) \big]_q
            &= \big[ B_r, [B_{r-1}, \ct_r^{-1}(B_{\t(r-1)})]_q \big]\\
            &= \big[ B_r, [\ct_r(B_{r-1}), B_{\t(r-1)}]_q \big]
    \end{align*}
    and hence we obtain
    \begin{align*}
        \ct_{r-1}\ct_r\ct_{r-1}\ct_r(B_r)
            &= q^{-2}\big[ B_r, [\ct_r(B_{r-1}), B_{\t(r-1)}]_q \big]\K{r}^2\K{r-1}K_{\varpi'_{r+1}}^2.
    \end{align*}
    Considering now the term $\ct_r\ct_{r-1}\ct_r\ct_{r-1}(B_r)$ we obtain
    \begin{align*}
        \ct_r\ct_{r-1}&\ct_r\ct_{r-1}(B_r)\\
            &= q^{-1/2}\ct_r\ct_{r-1}\ct_r([B_r,B_{r-1}]_q)\\
            &= q^{-3/2}\ct_r\ct_{r-1}( [B_r\K{r}K_{\varpi'_{r+1}},\ct_r(B_{r-1})]_q)\\
            &= q^{-2}\ct_r\big( \big[ [B_r,B_{r-1}]_q \K{r}\K{r-1}K_{\varpi'_{r+1}}, \ct_r^{-1}(B_{\t(r-1)}) \big]_q \big)\\
            &= q^{-2} \big[ [B_r\K{r}K_{\varpi'_{r+1}}, \ct_r(B_{r-1})]_q, B_{\t(r-1)} \big] \K{r}\K{r-1}K_{\varpi'_{r+1}}\\
            &= q^{-2}\big[ B_r, [\ct_r(B_{r-1}), B_{\t(r-1)}]_q \big]\K{r}^2\K{r-1}K_{\varpi'_{r+1}}^2\\
            &= \ct_{r-1}\ct_r\ct_{r-1}\ct_r(B_r)
    \end{align*}
    as required.
\end{proof}

\appendix
\section{Relations in \texorpdfstring{$\Beps$}{Bc}} \label{Appendix}
Many of the results in \cref{Sect:Braid_Action,Sec:Proof1,Sec:Proof2} require the use of additional relations which we provide here. We first give two useful relations that are used throughout this appendix.

Recall from \cref{Eqn:EJ+,Eqn:EJ-,Eqn:FJ+,Eqn:FJ-,Eqn:KJ,Eqn:EJFJ-oneelt} the elements $E_J^+, E_J^-, F_J^+, F_J^-$ and $K_J$ where $J \subset I$ is a subset of the form $J = \{a, a+1, \dotsc , b-1, b \}$ with $a \leq b$. Rewriting these elements using the Lusztig automorphisms, one sees that 
\begin{equation} \label{Eqn:EF-FE}
	E_J^+F_J^- - F_J^-E_J^+ = \frac{K_J - K_J^{-1}}{q-q^{-1}} = E_J^-F_J^+ - F_J^+E_J^-
\end{equation}
holds in $\Uqg$. Additionally, the $q$-commutator satisfies
\begin{equation} \label{Eqn:qcomm_trick}
	\big[ [x,y]_q, z\big]_q - \big[ x, [y,z]_q \big]_q = q\big[[x,z],y \big] 
\end{equation}
for all $x, y ,z \in \Uqg$.  

\subsection{Relations needed for the proof of Theorem \ref{Thm:ctr_Aut}}

   \begin{lemma} \label{Lem:simplecomms}
	The relations 
		\begin{align}
			\big[B_{r-1}, [F_X^+, \Z_r] \big]_q
				&= 0, \label{AppEqn:simplecomm1}\\
			\big[B_{r-1}, [F_X^+, \Z_{\t(r)}] \big]_q
				&= -(q-q^{-1})B_{r-1}\K{r}(K_X-K_X^{-1}) \label{AppEqn:simplecomm2}
		\end{align}
	hold in $\Beps$.
\end{lemma}

\begin{proof}
	Since $[B_{r-1}, F_X^+] = 0$ and $[B_{r-1}, \Z_r]_q = 0$ it follows that \cref{AppEqn:simplecomm1} holds. On the other hand, making use of the relation
		\begin{equation} \label{AppEqn:FX_Ztr_comm}
			[F_X^+, \Z_{\t(r)}] = q^{-1}(K_X-K_X^{-1})\K{r},
		\end{equation}
	which follows from \eqref{Eqn:EF-FE}, we obtain
		\begin{align*}
			\big[B_{r-1}, [F_X^+, \Z_{\t(r)}]\big]_q 
				&= q^{-1}[ B_{r-1}, (K_X-K_X^{-1})\K{r}]_q\\
				&= -(q-q^{-1})B_{r-1}(K_X-K_X^{-1})\K{r}
		\end{align*}
	as required.
\end{proof}
   
\begin{lemma} \label{AppLem:Alt_ctr}
	The relations
		\begin{align}
			\big[B_{r-1}, [B_r, [F_X^+, B_{\t(r)}]_q]_q \big]_q
				&= \big[[B_{r-1}, [B_r,F_X^+]_q]_q,B_{\t(r)}\big]_q \nonumber \\
				&\quad{}+q\cs_{\t(r)}B_{r-1}\K{r}(K_X - K_X^{-1}), \label{AppEqn:Newctr1}\\
			\big[B_{\t(r-1)}, [B_{\t(r)}, [F_X^-, B_r]_q]_q\big]_q
				&= \big[ [B_{\t(r-1)}, [B_{\t(r)},F_X^-]_q]_q, B_r \big]_q \nonumber \\
				&\quad{}+q\cs_rB_{\t(r-1)}\K{\t(r)}(K_X - K_X^{-1})\label{AppEqn:Newctr2}
		\end{align}
	hold in $\Beps$.
\end{lemma}

\begin{proof}
	By symmetry we only verify \cref{AppEqn:Newctr1}. By Equation \eqref{Eqn:qcomm_trick} we have
		\begin{align*}
			\big[B_r, [F_X^+, B_{\t(r)}]_q\big]_q
				&= \big[ [B_r,F_X^+]_q, B_{\t(r)}\big]_q + q\big[ F_X^+, [B_r,B_{\t(r)}]\big]\\
				&= \big[ [B_r, F_X^+]_q, B_{\t(r)}\big]_q - \frac{q\cs_{\t(r)}}{q-q^{-1}}[F_X^+,\Z_{\t(r)}] + \frac{q\cs_r}{q-q^{-1}}[F_X^+,\Z_r].
		\end{align*}
	Since $B_{r-1}$ commutes with $B_{\t(r)}$, Lemma \ref{Lem:simplecomms} implies
		\begin{align*}
			\big[B_{r-1}, [B_r,[F_X^+,&B_{\t(r)}]_q]_q \big]_q\\
				&= \big[ B_{r-1}, [[B_r, F_X]_q,B_{\t(r)}]_q\big]_q  -\frac{q\cs_{\t(r)}}{q-q^{-1}}\big[B_{r-1}, [F_X^+, \Z_{\t(r)}]\big]_q\\
				&= \big[ [B_{r-1}, [B_r,F_X^+]_q]_q, B_{\t(r)}\big]_q + q\cs_{\t(r)}B_{r-1}\K{r}(K_X-K_X^{-1})
		\end{align*}
	as required.
\end{proof}
Recall from \cref{Subsec:ctr_endo} the elements
\begin{align*}
	S 
		&= \big[ B_{r-1}, [B_r, F_X^+]_q\big]_q,\\
	S^{\t} 
		&= \big[ B_{\t(r-1)}, [B_{\t(r)}, F_X^-]_q\big]_q,\\        
	\Delta
		&= q\cs_{\t(r)}B_{r-1}\K{r}K_X,\\
	\Delta^{\t}
		&= q\cs_rB_{\t(r-1)}\K{\t(r)}K_X.
\end{align*}

For the remainder of this section, we provide relations that include the terms $S, S^{\t}, T$ and $T^{\t}$.

    \begin{lemma} \label{AppLem:Ztr_S_comm}
	The relations
		\begin{align} 
			\Z_{\t(r)} S 
				&= q^{-1}S \Z_{\t(r)} + (1-q^{-2}) [B_{r-1},B_r]_q \K{r}K_X, \label{AppEqn:ZtrS} \\
			\Z_r S^{\t}
				&= q^{-1} S^{\t} \Z_r + (1-q^{-2})[B_{\t(r-1)},B_{\t(r)}]_q \K{\t(r)}K_X, \label{AppEqn:ZrSt} 
		\end{align}
	hold in $\Beps$. 
\end{lemma}

    \begin{proof}
	We only prove that \cref{AppEqn:ZtrS} holds since the remaining checks are similar. By \cref{Eqn:EF-FE} and the relation $\Z_{\t(r)}B_i= q^{(\alpha_i,\alpha_{\t(r)}-\alpha_r)}B_i\Z_{\t(r)}$ for any $i \in I \setminus X$ we have
		\begin{align*}
			\Z_{\t(r)}S
				&= \Z_{\t(r)}\big[B_{r-1}, [B_r,F_X^+]_q\big]_q\\
				&= q^{-1}\big[ B_{r-1}, [B_r, \Z_{\t(r)}F_X^+]_{q^3}\big]\\
				&= q^{-1}\big[B_{r-1}, [B_r, F_X^+\Z_{\t(r)}- q^{-1}(K_X-K_X^{-1})\K{r}]_{q^3}\big]\\
				&= q^{-1}S\Z_{\t(r)}- q^{-2}\big[B_{r-1}, [B_r, (K_X-K_X^{-1})\K{r}]_{q^3}\big]\\
				&= q^{-1}S\Z_{\t(r)} + q^{-2}(q^2-1)[B_{r-1}, B_r\K{r}K_X]\\
				&= q^{-1}S\Z_{\t(r)} + (1-q^{-2})[B_{r-1},B_r]_q\K{r}K_X
		\end{align*}
	as required.    
	
\end{proof}

\begin{lemma} \label{AppLem:ST_comms}
	The relations
		\begin{align}
			B_rS &= SB_r, \label{AppEqn:BrS}\\
			B_{\t(r)}S^{\t} &= S^{\t}B_{\t(r)}, \label{AppEqn:BtrSt}
		\end{align}
	hold in $\Beps$.
\end{lemma}

\begin{proof}
	By symmetry we only verify \cref{AppEqn:BrS}. Using the relations
		\begin{equation*}
			p(B_{r},B_{r-1}) = p(B_r,F_X^+) = 0
		\end{equation*}
	we obtain
		\begin{align*}
			B_rS
				&= B_r\big[[B_{r-1},B_r]_q, F_X^+\big]_q\\
				&= B_rB_{r-1}B_rF_X^+ -qB_r^2B_{r-1}F_X^+ - qB_rF_X^+B_{r-1}B_r + q^2B_rF_X^+B_rB_{r-1}\\
				&= \tfrac{1}{q+q^{-1}}\big(B_r^2B_{r-1} + B_{r-1}B_r^2\big)F_X^+ - qB_r^2B_{r-1}F_X^+ - qB_rF_X^+B_{r-1}B_r\\
				&\quad{}+ \tfrac{q^2}{q+q^{-1}}\big(B_r^2F_X^+ + F_X^+B_r^2\big)B_{r-1}\\
				&= \tfrac{1}{q+q^{-1}}B_{r-1}B_r^2F_X^+ - qB_rF_X^+B_{r-1}B_r + \tfrac{q^2}{q+q^{-1}}F_X^+B_r^2B_{r-1}\\
				&= \tfrac{1}{q+q^{-1}}B_{r-1}\big( (q+q^{-1})B_rF_X^+B_r - F_X^+B_r^2 \big) -qB_rF_X^+B_{r-1}B_r\\
				&\quad{}+\tfrac{q^2}{q+q^{-1}}F_X^+\big((q+q^{-1})B_rB_{r-1}B_r - B_{r-1}B_r^2 \big)\\
				&= \big[[B_{r-1},B_r]_q, F_X^+ \big]_qB_r\\
				&= SB_r
		\end{align*}
	as required.
\end{proof}

\begin{lemma}
	The relations
		\begin{align}
			S B_{\t(r-1)}
				&= B_{\t(r-1)} S - q\cs_{r-1} \Z_{r-1} [B_r,F_X^+]_q, \label{AppEqn:SBtr-1}\\
				S^{\t} B_{r-1}
				&= B_{r-1} S^{\t} - q\cs_{\t(r-1)} \Z_{\t(r-1)} [B_{\t(r)},F_X^-]_q, \label{AppEqn:StBr-1}
		\end{align}
	hold in $\Beps$.
\end{lemma}

\begin{proof}
	We have
		\begin{align*}
			SB_{\t(r-1)}
				&= \big[B_{r-1}, [B_r,F_X^+]_q\big]_qB_{\t(r-1)}\\
				&= \big[B_{r-1}B_{\t(r-1)}, [B_r,F_X^+]_q\big]_q\\
				&= B_{\t(r-1)}S + \tfrac{1}{q-q^{-1}}\big[[\G_{r-1},B_r]_q, F_X^+\big]_q\\
				&= B_{\t(r-1)}S - q\cs_{r-1}\Z_{r-1}[B_r,F_X^+]_q
		\end{align*}
	as required. \Cref{AppEqn:StBr-1} is verified similarly.
\end{proof}

In the following lemma, we introduce the terms
\begin{align}
	\Omega^-
		&= \cs_{\t(r)}\cs_{\t(r-1)} F_X^- K_X\K{r}\Z_{\t(r-1)}, \label{AppEqn:O-}\\
	\Omega^+
		&= \cs_r\cs_{r-1} F_X^+ K_X\K{\t(r)}\Z_{r-1}, \label{AppEqn:O+}
\end{align}

\begin{lemma}
	The relations
		\begin{equation} 
			SS^{\t}-S^{\t}S = \Omega^- - \Omega^+, \label{AppEqn:SSt}
		\end{equation}
	hold in $\Beps$.
\end{lemma}

\begin{proof}
	Recall from \cref{Thm:cti} that $\ct_{r-1}$ is an algebra automorphism of $\Beps$ with inverse $\ct_{r-1}^{-1}$ given by \cref{Eqn:cti_inv}. We express $S$ and $S^{\t}$ using the algebra automorphisms $T_{w_X}$ and $\ct_{r-1}^{-1}$. In particular using Equations \eqref{Eqn:TwxBr}, \eqref{Eqn:TwxBtr} and \eqref{Eqn:cti_inv} we have
		\begin{align*}
			S &= (q\cs_{r-1})^{1/2}\ct_{r-1}^{-1} \circ T_{w_X}(B_r),\\
			S^{\t} &= (q\cs_{\t(r-1)})^{1/2}\ct_{r-1}^{-1} \circ T_{w_X}(B_{\t(r)}).
		\end{align*}
	Since $\cs_{r-1} = \cs_{\t(r-1)}$ we obtain
		\begin{align*}
			SS^{\t} - S^{\t}S 
				&= q\cs_{r-1}\ct_{r-1}^{-1} \circ T_{w_X}( B_r B_{\t(r)} - B_{\t(r)} B_r )\\
				&= \frac{q}{q-q^{-1}}\ct_{r-1}^{-1} \circ T_{w_X}( \cs_{r-1}\cs_r \Z_r - \cs_{\t(r)}\cs_{\t(r-1)} \Z_{\t(r)}).
		\end{align*}
	Equation \eqref{TwXEX} implies that
		\begin{align*}
			T_{w_X}(\Z_r) &= (1-q^{-2}) F_X^+ K_X\K{\t(r)},\\
			T_{w_X}(\Z_{\t(r)}) &= (1-q^{-2}) F_X^- K_X\K{r}.
		\end{align*}
	Hence we obtain
		\begin{align*}
			SS^{\t} - S^{\t}S 
				&= \ct_{r-1}^{-1} \big( \cs_{r-1}\cs_r F_X^+ K_X\K{\t(r)} - \cs_{\t(r)}\cs_{\t(r-1)} F_X^- K_X\K{r} \big)\\
				&= \cs_{\t(r)}\cs_{\t(r-1)} F_X^- K_X\K{r} \Z_{\t(r-1)} - \cs_{r-1}\cs_r F_X^+ K_X\K{\t(r)} \Z_{r-1} \\
				&= \Omega^- - \Omega^+
		\end{align*}
	as required. 
\end{proof}

\begin{lemma} \label{Lemma:Tech_Calc_1}
	The relation
		\begin{align} \label{Eqn:Tech_Calc_1}
			\begin{split}
				\big[ [S,B_{\t(r)}]_q, [S^{\t},B_r]_q  \big] &= - \big[ \Delta, [S^{\t},B_r]_q \big] - \big[ [S,B_{\t(r)}]_q, \Delta^{\t} \big] \\
				&\quad{}- \frac{q \cs_r \cs_{\t(r)} }{ q-q^{-1} } \big( K_X - K_X^{-1} ) K_X\G_{r-1}
			\end{split}
		\end{align}
	holds in $\Beps$.
\end{lemma}

\begin{proof}
	We first use \cref{AppEqn:BrS,AppEqn:BtrSt,AppEqn:SSt} to rewrite each term of $[S, B_{\t(r)}]_q[S^{\t}, B_r]_q$. In particular we have
		\begin{align*}
			SB_{\t(r)}S^{\t}B_r 
				&= SS^{\t}B_{\t(r)}B_r\\
				&= \big( S^{\t}S + \O \big) \big( B_rB_{\t(r)} - \tfrac{1}{q-q^{-1}}\G_r \big)\\
				&= S^{\t}SB_rB_{\t(r)} - \tfrac{1}{q-q^{-1}}S^{\t}S\G_r + ( \O )  B_rB_{\t(r)}\\ 
				&\quad{}- \tfrac{1}{q-q^{-1}} (\O )\G_r.
		\end{align*}
	Similarly, one finds that
		\begin{align*}
			S B_{\t(r)} B_r S^{\t}
				&= B_r S^{\t} S B_{\t(r)} + B_r ( \O ) B_{\t(r)} - \tfrac{1}{q-q^{-1}} S \G_r S^{\t},\\
				B_{\t(r)} S S^{\t} B_r 
				&= S^{\t} B_r B_{\t(r)} S + B_{\t(r)} ( \O )B_r - \tfrac{1}{q-q^{-1}}S^{\t} \G_r S,\\
				B_{\t(r)} S B_r S^{\t}
				&= B_r S^{\t} B_{\t(r)} S + B_{\t(r)} B_r ( \O ) - \tfrac{1}{q-q^{-1}}\G_r SS^{\t}\\
				&\quad{}+ \tfrac{1}{q-q^{-1}}\G_r( \O ).
		\end{align*}
	Combining these four expressions we obtain 
		\begin{align} \label{Eqn:StepOne}
			\begin{split}
			\big[[S, B_{\t(r)}]_q, [S^{\t},B_r]_q\big]
				&= -\tfrac{1}{q-q^{-1}} \big[ S^{\t}, [S,\G_r]_q \big]_q + \big[ [\O, B_r]_q, B_{\t(r)}\big]_q\\
				&\qquad{}- \tfrac{1}{q-q^{-1}}\big[ \O, \G_r \big]_{q^2}.
			\end{split}
		\end{align}
	We now consider the term $\big[S^{\t}, [S,\G_r]_q \big]_q$ in more detail. Using \cref{AppEqn:ZtrS,AppEqn:ZrSt} we have
		\begin{align*}
			S^{\t} \G_r S
				&= \cs_r S^{\t} \Z_r S - \cs_{\t(r)} S^{\t} \Z_{\t(r)} S\\
				&= \cs_r \big( q\Z_rS^{\t} - (q-q^{-1})[B_{\t(r-1)}, B_{\t(r)}]_q K_X\K{\t(r)} \big) S\\
				&\quad{} - \cs_{\t(r)} S^{\t} \big( q^{-1} S \Z_{\t(r)} + (1-q^{-2}) [B_{r-1},B_r]_q K_X\K{r} \big)\\
				&= q\cs_r \Z_r S^{\t} S - q^{-1}\cs_{\t(r)} S^{\t} S \Z_{\t(r)} - (q-q^{-1}) [\Delta^{\t}, B_{\t(r)}] S\\
				&\quad{}- (1-q^{-2}) S^{\t} [\Delta,B_r].
		\end{align*}
	We similarly obtain
		\begin{align*}
			S \G_r S^{\t}
				&= q^{-1}\cs_r S S^{\t} \Z_r - q\cs_{\t(r)} \Z_{\t(r)} S S^{\t} + (1-q^{-2}) S [\Delta^{\t}, B_{\t(r)}]\\
				&\quad{}+ (q-q^{-1}) [\Delta, B_r] S^{\t}.
		\end{align*}
	It hence follows from this and \cref{AppEqn:SSt} that
	\begin{align} \label{Eqn:StepTwo}
		\begin{split}
		\big[S^{\t}, [S,\G_r]_q\big]_q
			&= (q-q^{-1})\big[S^{\t}, [\Delta, B_r]\big]_q - (q-q^{-1})\big[ S, [\Delta^{\t}, B_{\t(r)}]\big]_q\\
			&\quad{}- c_r(\O)\Z_r + q^2c_{r}\Z_{r}(\O). 
		\end{split}
	\end{align}
	We now consider the elements $\big[ S^{\t}, [\Delta,B_r] \big]_q$ and $\big[ S, [\Delta^{\t}, B_{\t(r)}] \big]_q$ and write them in the form that appears in \cref{Eqn:Tech_Calc_1}. By \cref{AppEqn:StBr-1} it follows that
		\begin{align*}
			\big[ S^{\t}, [\Delta, B_r] \big]_q 
				&= S^{\t}\Delta B_r - S^{\t}B_r\Delta - q\Delta B_rS^{\t} + q B_r\Delta S^{\t}\\
				&= \big( \Delta S^{\t} + [ \Omega^-, B_{\t(r)}]_q\big) B_r - S^{\t}B_r\Delta - q\Delta B_rS^{\t}\\
				&\quad{}+ qB_r\big(S^{\t}\Delta - [\Omega^-, B_{\t(r)}]_q \big)\\
				&= \big[ \Delta, [S^{\t}, B_r]_q \big] + \big[ [\Omega^-, B_{\t(r)}]_q, B_r \big]_q.
		\end{align*}
	Similarly, by \cref{AppEqn:SBtr-1} we have
		\begin{align*}
			\big[S, [\Delta^{\t}, B_{\t(r)}] \big]_q
				&= -\big[ [S,B_{\t(r)}], \Delta^{\t} \big] + \big[ [\Omega^+, B_r]_q, B_{\t(r)}\big]_q.
		\end{align*}
	Substituting these two expressions into \cref{Eqn:StepTwo} gives
		\begin{align} \label{Eqn:StepThree}
			\begin{split}
			\big[S^{\t}, [S,\G_r]_q \big]_q
				&= (q-q^{-1})\big[ \Delta, [S^{\t},B_r]_q \big] + (q-q^{-1})\big[[S,B_{\t(r)}], \Delta^{\tau}\big]\\
				&\quad{}+ (q-q^{-1})\big[[\Omega^-, B_{\t(r)}]_q, B_r\big]_q - (q-q^{-1})\big[ [\Omega^+, B_r]_q, B_{\t(r)}\big]_q\\
				&\qquad{} -\cs_r( \O )\Z_r + q^2\cs_{r}\Z_{r} (\O).
			\end{split}
		\end{align}
	We substitute \cref{Eqn:StepThree} into \cref{Eqn:StepOne}. Noting that
		\begin{equation*}
			\big[ [\Omega^-, B_r]_q, B_{\t(r)}\big]_q - \big[ [\Omega^-, B_{\t(r)}]_q, B_r\big]_q = \tfrac{1}{q-q^{-1}}[\Omega^-, \G_r]_{q^2},
		\end{equation*}
	we obtain 
		\begin{align*}
			\big[ [S,B_{\t(r)}]_q, [S^{\t},B_r]_q \big]
				&= -\big[\Delta, [S^{\t},B_r]_q \big] - \big[ [S,B_{\t(r)}]_q, \Delta^{\t} \big]\\
				&\quad{}+ \tfrac{1}{q-q^{-1}} \big( c_r[\Omega^-, \Z_r]_{q^2} - c_{\t(r)}[\Omega^+, \Z_{\t(r)}]_{q^2} \big).
		\end{align*}    
	Using \cref{Eqn:EF-FE} we compute $[\Omega^-, \Z_r]_{q^2}$ and $[\Omega^+, \Z_{\t(r)}]_{q^2}$. This gives
		\begin{align*}
			[\Omega^-, \Z_r]_{q^2} &= q\cs_{\t(r)}\cs_{\t(r-1)}\big(K_X - K_X^{-1}\big)K_X\Z_{\t(r-1)},\\
			[\Omega^+, \Z_{\t(r)}]_{q^2} &= q\cs_{r}\cs_{r-1}\big(K_X - K_X^{-1}\big)K_X\Z_{r-1}.
		\end{align*}
	It hence follows that
		\begin{align*}
			\big[ [S,B_{\t(r)}]_q, [S^{\t},B_r]_q \big]
				&= -\big[\Delta, [S^{\t},B_r]_q \big] - \big[ [S,B_{\t(r)}]_q, \Delta^{\t} \big]\\
				&\quad{}-\frac{q\cs_r\cs_{\t(r)}}{q-q^{-1}}\big(K_X-K_X^{-1})K_X\G_{r-1}
		\end{align*} 
	as required.
\end{proof}

    \begin{lemma} \label{Lem:QSerre2_Rel1}
	The relation 
		\begin{equation} \label{Eqn:QSerre2_Rel1}
			\Z_r[S, B_{\t(r)}]_q = q^{-1}[S,B_{\t(r)}]_q\Z_r + q^{-2}(q-q^{-1})\Delta\Z_r
		\end{equation}
	holds in $\Beps$.
\end{lemma}

\begin{proof}
	The difficulty in the proof comes from that fact that the element $\Z_r$ contains $E_X^+$ as a factor, and there is generally no simple way to commute $E_X^+$ with $F_X^+$. The idea is to verify that \cref{Eqn:QSerre2_Rel1} holds if the algebra automorphism $T_{w_X}^{-1}$ is applied to both sides. More precisely, using \eqref{Eqn:TwXFX}, \eqref{TwXEX}, \eqref{Eqn:TwxBr} and \eqref{Eqn:TwxiBtr} we have
		\begin{align*}
			[S,B_{\t(r)}]_q &= T_{w_X}\big( \big[ B_{r-1}, [B_r, [F_X^+, B_{\t(r)}]_q]_q \big]_q \big),\\
			\Z_r &= q^2(1-q^{-2})T_{w_X} \big( F_X^+K_X^{-1}\K{\t(r)} \big),\\
			\Delta &= q\cs_{\t(r)}T_{w_X} \big( B_{r-1}\K{r}K_X^{-1} \big)
		\end{align*}
	and hence verifying \cref{Eqn:QSerre2_Rel1} is equivalent to showing that
		\begin{align}
			\begin{split}
				F_X^+\big[ B_{r-1}, [B_r, [F_X^+, B_{\t(r)}]_q]_q \big]_q
					&= \big[ B_{r-1}, [B_r, [F_X^+, B_{\t(r)}]_q]_q \big]_qF_X^+\\
					&\quad{}+ (q-q^{-1})\cs_{\t(r)}B_{r-1}\K{r}K_X^{-1}F_X^+ \label{Eqn:Zr[S,Btr]check}
			\end{split}
		\end{align}
	holds in $\Beps$.
	Using the relations 
	\[p(F_X^+,B_r) = p(F_X^+, B_{\t(r)}) = 0\] 
	we can commute $F_X^+$ through $[B_r, [F_X^+, B_{\t(r)}]_q]_q$. This gives
		\begin{align*}
			F_X^+ [B_r, [&F_X^+, B_{\t(r)}]_q]_q\\
				&= F_X^+B_rF_X^+B_{\t(r)} - qF_X^+B_rB_{\t(r)}F_X^+ - qF_X^{+2}B_{\t(r)}B_r+ q^2F_X^+B_{\t(r)}F_X^+B_r\\
				&= \tfrac{1}{q+q^{-1}}\big(F_X^{+2}B_r + B_rF_X^{+2} \big)B_{\t(r)} - qF_X^+B_rB_{\t(r)}F_X^+ -qF_X^{+2}B_{\t(r)}B_r \\
				&\quad{}+ \tfrac{q^2}{q+q^{-1}}\big( F_X^{+2}B_{\t(r)} + B_{\t(r)}F_X^{+2}\big)B_r\\
				&= \tfrac{1}{q^2-q^{-2}}F_X^{+2}\G_r + \tfrac{1}{q+q^{-1}}B_r\big( (q+q^{-1})F_X^+B_{\t(r)}F_X^+ - B_{\t(r)}F_X^{+2}\big)\\
				&\quad{}-qF_X^+B_rB_{\t(r)}F_X^+ + \tfrac{q^2}{q+q^{-1}}B_{\t(r)}\big( (q+q^{-1})F_X^+B_rF_X^+ - B_rF_X^{+2}\big)\\
				&= [B_r, [F_X^+, B_{\t(r)}]_q]_qF_X^+ + \tfrac{1}{q^2-q^{-2}}\big( F_X^{+2}\G_r - (1 + q^2)F_X^+\G_rF_X^+ + q^2\G_rF_X^{+2}\big)
		\end{align*}
	where the third equality follows from observing that terms beginning with $F_X^{+2}$ simplify. Since $B_{r-1}$ commutes with $F_X^+$, and 
	\[ [B_{r-1}, \G_r]_q = (q^2-1)\cs_{\t(r)}B_{r-1}\Z_{\t(r)} \]
	it follows that
		\begin{align*}
			F_X^+\big[B_{r-1}, [&B_r,[F_X^+,B_{\t(r)}]_q]_q \big]_q - \big[B_{r-1}, [B_r,[F_X^+,B_{\t(r)}]_q]_q \big]_qF_X^+\\
				&=\frac{q}{q+q^{-1}}\cs_{\t(r)}B_{r-1}\big( F_X^{+2}\Z_{\t(r)} - (1+q^2)F_X^+\Z_{\t(r)}F_X^+ + q^2\Z_{\t(r)}F_X^{+2}\big).
		\end{align*}
	The relation
	\[ F_X^+\Z_{\t(r)} = \Z_{\t(r)}F_X^+ + q^{-1}(K_X - K_X^{-1})\K{r} \]
	implies that
		\begin{align*}
			F_X^{+2}\Z_{\t(r)} &- (1+q^2)F_X^+\Z_{\t(r)}F_X^+ + q^2\Z_{\t(r)}F_X^{+2}\\
				&= q^{-1}F_X^+(K_X-K_X^{-1})\K{r} - q(K_X-K_X^{-1})\K{r}F_X^+\\
				&= q^{-1}(q^2-q^{-2})K_X^{-1}\K{r}F_X^+.
		\end{align*}
	Hence \eqref{Eqn:Zr[S,Btr]check} holds as required.
\end{proof}

    \begin{lemma} \label{Lemma:QSerre2_Rel2}
	The relation
	\begin{equation} \label{Eqn:QSerre2_Rel2}
		S[S, B_{\t(r)}]_q = q^{-1}[S,B_{\t(r)}]_qS - q^{-2}(q^2-q^{-2})S\Delta
	\end{equation}
	holds in $\Beps$.
\end{lemma}

\begin{proof}
	As in the proof of \cref{Lem:QSerre2_Rel1} we verify a relation that is equivalent to \cref{Eqn:QSerre2_Rel2}. The difference here is that we additionally use the algebra automorphism $\ct_{r-1}$ from \cref{Eqn:cti}. In particular using \eqref{Eqn:TwxBr}, \eqref{Eqn:TwxBtr} and \eqref{Eqn:cti_inv} we have
		\begin{align*}
			S &= (q\cs_{r-1})^{1/2} T_{w_X} \circ \ct_{r-1}^{-1}(B_r),\\
			B_{\t(r)} &= (q\cs_{\t(r-1)})^{-1/2} T_{w_X} \circ \ct_{r-1}^{-1}\big( [F_X^+, [B_{\t(r)}, B_{\t(r-1)}]_q]_q \big),\\
			\Delta &= \cs_{\t(r)} T_{w_X} \circ \ct_{r-1}^{-1} \big( B_{\t(r-1)}\K{r}K_X^{-1}  \big)
		\end{align*}
	and hence we are done if we show that
		\begin{align*}
			B_r\big[ B_r, [F_X^+, [B_{\t(r)}, B_{\t(r-1)}]_q]_q\big]_q 
				&= q^{-1}\big[ B_r, [F_X^+, [B_{\t(r)}, B_{\t(r-1)}]_q]_q\big]_q B_r \\
				&\quad{}-q^{-2}(q^2-q^{-2})\cs_{\t(r)}B_rB_{\t(r-1)}\K{r}K_X^{-1}.
		\end{align*}
	Noting that
		\begin{align}
			B_r[B_{\t(r)}, B_{\t(r-1)}]_q
				&= [B_{\t(r)},B_{\t(r-1)}]_qB_r + \tfrac{1}{q-q^{-1}}[\G_r, B_{\t(r-1)}]_q \nonumber\\
				&= [B_{\t(r)},B_{\t(r-1)}]_qB_r + q\cs_{\t(r)}\Z_{\t(r)}B_{\t(r-1)} \label{Eqn:QSerre2_Part1}
		\end{align}
	one calculates
		\begin{align*}
			\big[ B_r, [F_X^+, [B_{\t(r)}, B_{\t(r-1)}]_q]_q\big]_q
				&= \big[ [B_r,F_X^+]_q, [B_{\t(r)}, B_{\t(r-1)}]_q \big]_q\\
				&\quad{}+ q^2c_{\t(r)}[F_X^+, \Z_{\t(r)}]B_{\t(r-1)}.
		\end{align*}
	Using this, the relation $B_r[B_r,F_X^+]_q = q^{-1}[B_r,F_X^+]_qB_r$ and \cref{Eqn:QSerre2_Part1} we have
		\begin{align*}
			B_r \big[&B_r, [F_X^+, [B_{\t(r)}, B_{\t(r-1)}]_q]_q \big]_q\\
				&= B_r\big[ [B_r,F_X^+]_q, [B_{\t(r)},B_{\t(r-1)}]_q \big]_q + q^2\cs_{\t(r)}B_r[F_X^+,\Z_{\t(r)}]B_{\t(r-1)}\\
				&= q^{-1}\big[ [B_r,F_X^+]_q, [B_{\t(r)},B_{\t(r-1)}]_q \big]_qB_r + \cs_{\t(r)}B_r[F_X^+, \Z_{\t(r)}]B_{\t(r-1)}\\ 
				&\quad{}- q^3\cs_{\t(r)}[F_X^+, \Z_{\t(r)}]B_rB_{\t(r-1)} + q^2\cs_{\t(r)}B_r[F_X^+, \Z_{\t(r)}]B_{\t(r-1)}\\
				&= q^{-1}\big[B_r, [F_X^+, [B_{\t(r)}, B_{\t(r-1)}]_q]_q \big]_qB_r+ (1 \!+ \!q^2)\cs_{\t(r)}\big[ B_r, [F_X^+, \Z_{\t(r)}] \big]_qB_{\t(r-1)}.
		\end{align*}
	Since $[F_X^+, \Z_{\t(r)}] = q^{-1}(K_X-K_X^{-1})\K{r}$ we have
		\begin{align*}
			\big[ B_r, [F_X^+, \Z_{\t(r)}]\big]_q
				&= q^{-1}[ B_r, (K_X-K_X^{-1})\K{r}]_q\\
				&= -q^{-1}(1-q^{-2})B_rK_X^{-1}\K{r}
		\end{align*}
	and hence we obtain
		\begin{align*}
			B_r \big[B_r, [F_X^+, [B_{\t(r)}, B_{\t(r-1)}]_q]_q \big]_q
				&= q^{-1}\big[B_r, [F_X^+, [B_{\t(r)}, B_{\t(r-1)}]_q]_q \big]_qB_r\\
				&\quad{}-q^{-2}(q^2-q^{-2})\cs_{\t(r)}B_rB_{\t(r-1)}K_X^{-1}\K{r}
		\end{align*}
	as required.
\end{proof}

    \begin{lemma} \label{AppLem:ctrinv}
	The relation
	\begin{align} 
		\big[F_X^-, [S, \Z_r]_q \big]_q &= (q^2-1)S\K{\t(r)}K_X^{-1}, \label{AppEqn:ctrinv1}
	\end{align}
	holds in $\Beps$.
\end{lemma}

\begin{proof}
	Observe that
	\[ \big[F_X^-, [S,\Z_r]_q \big]_q = \big[ [F_X^-,S]_q,\Z_r \big]_q + q\big[S, [F_X^-, \Z_r] \big] \]
	holds by \cref{Eqn:qcomm_trick}. Since $S = T_{w_X}([B_{r-1},B_r]_q)$ and $F_X^- = -T_{w_X}(E_X^-K_X)$ it follows that
	\begin{align*}
		[F_X^-,S]_q &=-T_{w_X}\big(\big[E_X^-K_X, [B_{r-1},B_r]_q \big]_q \big)\\
			&= 0
	\end{align*}
	since $[E_X^-K_X,B_r]_q= 0$. Further we have 
	\begin{align*}
		[F_X^-,\Z_r] &= -(1-q^{-2})[F_X^-, E_X^+\K{\t(r)}]\\
			&= -(1-q^{-2})[F_X^-,E_X^+]\K{\t(r)}\\
			&=q^{-1}(K_X - K_X^{-1})\K{\t(r)}.
	\end{align*}
	Hence we obtain 
	\begin{align*}
		\big[F_X^-,[S,\Z_r]_q \big]_q
			&= q\big[ S, [F_X^-, \Z_r] \big]\\
			&= \big[S, K_X- K_X^{-1}]_q\K{\t(r)}\\
			&= (q^2-1)S\K{\t(r)}K_X^{-1}
	\end{align*}
	as required.
\end{proof}

\subsection{Relations needed for the proof of Theorem \ref{Thm:ctr_braid}}
    \begin{lemma} \label{AppLem:1}
	The relation
	\begin{equation} \label{AppEqn:1}
		\big[\ct_{r-1}(B_{r-1}), [F_X^+, [B_{\t(r)}, B_{\t(r-1)}]_q]_q \big] = 0
	\end{equation}
	holds in $\Beps$.
\end{lemma}

\begin{proof}
	Since $\ct_{r-1}(B_{r-1}) = q^{-1}B_{\t(r-1)}\K{\t(r-1)}$ it follows that $\ct_{r-1}(B_{r-1})$ commutes with $F_j$ for $j \in X$. Further, \cref{Eqn:BiRel4} implies that
	\[ B_{\t(r-1)}[B_{\t(r)}, B_{\t(r-1)}]_q = q[B_{\t(r)}, B_{\t(r-1)}]_qB_{\t(r-1)} \]
	and hence $\ct_{r-1}(B_{r-1})$ commutes with $[B_{\t(r)}, B_{\t(r-1)}]_q$. The result follows from this.
\end{proof}

\begin{lemma} \label{AppLem:2}
	For any $i \in I \setminus (X \cup \{r,\t(r)\})$ the relations
	\begin{align}
		\big[ B_{\t(i)}\K{\t(i)}, [B_{i \pm 1},B_i]_q \big]_q &= q^2\cs_iB_{i \pm 1}, \label{AppEqn:2}\\
		\big[ [B_i,B_{i \pm 1}]_q, B_{\t(i)}\K{i}\big]_q &= \cs_iB_{i \pm 1}\label{AppEqn:3}
	\end{align}
	hold in $\Beps$.
\end{lemma}

\begin{proof}
	The relations follow immediately by applying the automorphisms $\ct_i$ and $\ct_i^{-1}$ to 
	\begin{align*}
		\ct_i^{-1}(B_{i \pm 1}) &= (q\cs_i)^{-1/2}[B_i,B_{i \pm 1}]_q,\\
		\ct_i(B_{i \pm 1}) &= (q\cs_i)^{-1/2}[B_{i \pm 1}, B_i]_q,
	\end{align*}
	respectively.
\end{proof}
In the following Lemma, which is used in the proof of Proposition \ref{Prop:Braid-j=r-1}, we make use of the fact that $\ct_r$ is an algebra automorphism of $\Beps$, see Theorem \ref{Thm:ctr_Aut} and Section \ref{Sec:Proof1}. 
\begin{lemma} \label{AppLem:3}
	The relation 
	\begin{equation} \label{AppEqn:4}
		\ct_{r-1}\ct_r\ct_{r-1}\ct_r(B_{r-1}) = \ct_{r-1}(B_{\t(r-1)})
	\end{equation}            
	holds in $\Beps$.
\end{lemma}

\begin{proof}
	Calculating directly we have
	\begin{align*}
		\ct_{r-1}\ct_r(B_{r-1})
			&= C\big( \big[\ct_{r-1}(B_{r-1}), [\ct_{r-1}(B_{r}), [F_X^+, \ct_{r-1}(B_{\t(r)})]_q]_q \big]_q\\ 
			&\quad{}+ q\cs_r\ct_{r-1}(B_{r-1}\K{r})K_X^{-1} \big)\\
			&= C\big(q^{-2}\cs_r^{-1}\big[B_{\t(r-1)}\K{\t(r-1)}, [[B_r,B_{r-1}]_q,[F_X^+, [B_{\t(r)},B_{\t(r-1)}]_q]_q]_q \big]_q\\
			&\quad{} +\cs_rB_{\t(r-1)}\K{r}K_X^{-1} \big)\\
			&= \ct_{r}^{-1}(B_{\t(r-1)})
	\end{align*}
	where the last equality is obtained using \cref{AppEqn:1,AppEqn:2}. This implies that
	\[ \ct_{r-1}\ct_r\ct_{r-1}\ct_r(B_{r-1}) = \ct_{r-1}(B_{\t(r-1)})\]
	as required.
\end{proof}

\begin{lemma} \label{AppLem:4}
	The relation
	\begin{equation} \label{AppEqn:5}
		\big[ B_{r-1}, [B_r, [F_X^+, [B_{\t(r)}, B_{\t(r-1)}]_q]_q]_q \big]_q
			= \big[[B_{r-1}, [B_r, [F_X^+, B_{\t(r)}]_q]_q]_q, B_{\t(r-1)}\big]_q
	\end{equation}
	holds in $\Beps$. 
\end{lemma}

\begin{proof}
	First observe that since $B_{\t(r-1)}$ commutes with $B_r$ and $F_X^+$ we have
	\[\big[B_r, [F_X^+, [B_{\t(r)}, B_{\t(r-1)}]_q]_q\big]_q
	= \big[ [B_r, [F_X^+, B_{\t(r)}]_q]_q, B_{\t(r-1)}\big]_q.\]
	To shorten notation, let $Y = [B_r,[F_X^+, B_{\t(r)}]_q]_q$. Recall from \cref{Eqn:BiRel3} that
	\[ B_{r-1}B_{\t(r-1)} - B_{\t(r-1)}B_{r-1} = (q-q^{-1})^{-1}(\cs_{r-1}\Z_{r-1} - \cs_{\t(r-1)}\Z_{\t(r-1)}) \]
	where $\Z_{r-1} = -\K{\t(r-1)}$ and $\Z_{\t(r-1)} = -\K{r-1}$. Then $Y$ commutes with both $\Z_{r-1}$ and $\Z_{\t(r-1)}$. Recalling the notation $\G_i$ from \cref{Eqn:Gi} we hence have
		\begin{align*}
			\big[B_{r-1}, &[Y, B_{\t(r-1)}]_q \big]_q\\
				&= B_{r-1}YB_{\t(r-1)} - qB_{r-1}B_{\t(r-1)}Y - qYB_{\t(r-1)}B_{r-1} + q^2B_{\t(r-1)}YB_{r-1}\\
				&= B_{r-1}YB_{\t(r-1)} - q(B_{\t(r-1)}B_{r-1} + (q-q^{-1})^{-1}\G_{r-1}Y)\\
				&\quad{}- qY(B_{r-1}B_{\t(r-1)} - (q-q^{-1})^{-1}\G_{r-1}) +q^2B_{\t(r-1)}YB_{r-1} \\
				&= \big[[B_{r-1},Y]_q, B_{\t(r-1)}\big]_q
		\end{align*}
	as required.
\end{proof}

\begin{corollary} \label{AppCor:5}
	The element $[B_{r-1}, \ct_r^{-1}(B_{\t(r-1)})]_q$ is $\ct_r$-invariant i.e.
	\begin{equation} \label{AppEqn:6}
		[B_{r-1}, \ct_r^{-1}(B_{\t(r-1)})]_q = [\ct_r(B_{r-1}),B_{\t(r-1)}]_q.
	\end{equation}
\end{corollary}

\begin{proof}
	The result follows immediately from \cref{AppLem:4} and the fact that
	\[ [B_{r-1}, B_{\t(r-1)}\K{r}K_X^{-1}]_q = q[B_{r-1}\K{r}K_X^{-1}, B_{\t(r-1)}]_q. \]
\end{proof}

\subsection{Relations needed for the proof of Theorem \ref{Thm:MAIN}}

    \begin{lemma} \label{AppLem:6}
        For any $j \in X \setminus \{r+1,\t(r+1)\}$ the relation
            \begin{equation} \label{AppEqn:7}
                T_j(F_X^+) = F_X^+
            \end{equation}
        holds.
    \end{lemma}
    
    \begin{proof}
        Since $j \in X \setminus \{r+1, \t(r+1)\}$ we assume that $|X| \geq 3$. Recalling that 
            \[ F_X^+ = \big[F_{r+1}, [F_{r+2}, \dotsc,  [F_{\t(r+2)}, F_{\t(r+1)}]_q \dotsc ]_q \big]_q \]
        the result follows since for any $j \in X \setminus \{r+1, \t(r+1)\}$ we have
            \begin{align*}
                T_j([F_{j-1}, [F_j,F_{j+1}]_q]_q) 
                    &= T_j([F_{j-1}, T_j^{-1}(F_{j+1})]_q)\\
                    &= [ [F_{j-1}, F_j]_q,F_{j+1}]_q\\
                    &= [F_{j-1}, [F_j, F_{j+1}]_q]_q.
            \end{align*}
    \end{proof}

    \begin{lemma} \label{AppLem:7}
        The relation
            \begin{align}
                \begin{split}
                T_{r+1}\big( \big[B_{r-1}, [B_r, [F_X^+, B_{\t(r)}]_q]_q\big]_q\big)
                    &=  [B_{r-1}, [B_r, [F_X^+, B_{\t(r)}]_q]_q]_q\\
                    &\quad{} + q\cs_{\t(r)}B_{r-1}\K{r}(K_{r+1}^{-1}-K_{r+1})K_{X\setminus \{r+1\}}^{-1}\label{AppEqn:8}
                \end{split}
            \end{align}
        holds in $\Beps$.
    \end{lemma}
    
    \begin{proof}
    	First suppose that $X = \{r+1\}$. By Equation \eqref{Eqn:TiBj} we have $[F_{r+1}, B_{r+2}]_q = T_{r+1}^{-1}(B_{r+2})$ and $T_{r+1}(B_r) = [B_r, F_{r+1}]_q$.
    	Hence
    		\begin{equation} \label{Eqn:T(r+1)chain1}
    			T_{r+1}\big( \big[ B_{r-1}, [B_r, [F_{r+1}, B_{r+2}]_q ]_q \big]_q \big) 
    				=	\big[ B_{r-1}, [ [ B_r, F_{r+1}]_q, B_{r+2} ]_q \big]_q.
    		\end{equation}
    	By Equation \eqref{Eqn:qcomm_trick} we have
    		\begin{align*}
	    		\big[ [B_r, F_{r+1}]_q, B_{r+2} \big]_q
	    			&=	\big[ B_r, [F_{r+1}, B_{r+2}]_q \big]_q + q\big[ [B_r, B_{r+2}], F_{r+1} \big]\\
	    			&=	\big[ B_r, [F_{r+1}, B_{r+2}]_q \big]_q + \frac{q}{q-q^{-1}}[ \cs_r\Z_r - \cs_{r+2}\Z_{r+2}, F_{r+1}].
    		\end{align*}
    	Since $[B_{r-1}, \K{r+2}]_q = 0$ it follows that $\big[ B_{r-1}, [\Z_r, F_{r+1}] \big]_q = 0$. On the other hand we have
    		\begin{align*}
	    		[\Z_{r+2}, F_{r+1}] 
	    			&=	-(1-q^{-2}) [E_{r+1}\K{r}, F_{r+1}]\\
	    			&=	-(1-q^{-2}) \K{r}[E_{r+1}, F_{r+1}]\\
	    			&=	-q^{-1}\K{r}(K_{r+1} - K_{r+1}^{-1})
    		\end{align*}
    	and hence 
    		\begin{align*}
	    		\big[B_{r-1}, [\Z_{r+2}, F_{r+1}] \big]_q
	    			&=	-q^{-1}[B_{r-1}, \K{r}]_q (K_{r+1} - K_{r+1}^{-1})\\
	    			&=	(q-q^{-1}) B_{r-1}\K{r}(K_{r+1} - K_{r+1}^{-1}).
    		\end{align*}
    	It follows that 
    		\begin{align*}
    			\big[B_{r-1}, [ [B_r, F_{r+1}]_q, B_{r+2}]_q \big]_q 
    				&= \big[B_{r-1}, [B_r, [F_{r+1}, B_{r+2}]_q ]_q \big]_q\\
    					 &\quad{}+ q\cs_{r+2}B_{r-1}\K{r}(K_{r+1}^{-1} - K_{r+1}).
    		\end{align*}
    	The result follows by substituting this into \eqref{Eqn:T(r+1)chain1}.
    	
    	We now consider the case $|X| > 1$.	Let $Y = X \setminus \{r+1\}$. Observing that $[F_X^+, B_{\t(r)}]_q = T_{r+1}^{-1}([F_Y^+, B_{\t(r)}]_q)$ we have
        	\begin{equation} \label{Eqn:T(r+1)chain2}
             	T_{r+1}\big( \big[B_{r-1}, [B_r, [F_X^+, B_{\t(r)}]_q]_q\big]_q\big)
            		= \big[ B_{r-1}, [[B_r,F_{r+1}]_q, [F_Y^+, B_{\t(r)}]_q]_q \big]_q.
        	\end{equation}
        Since $F_Y^+$ commutes with $B_r$, Equation \eqref{Eqn:qcomm_trick} implies
         \begin{align*}
            \big[[B_r, F_{r+1}]_q, [F_Y^+, B_{\t(r)}]_q \big]_q &- [B_r,[F_X^+,B_{\t(r)}]_q]_q \\
                &= \frac{q}{q-q^{-1}}\big[[F_Y^+,\cs_r\Z_r -\cs_{\t(r)}\Z_{\t(r)} ]_q,F_{r+1} \big].
         \end{align*}
         Since $[B_{r-1}, \K{\t(r)}]_q=0$ it follows that
            \[ \big[ B_{r-1}, [[F_Y^+, \Z_r]_q, F_{r+1}]\big]_q = 0. \]
         On the other hand, using \cref{Eqn:EF-FE} it follows that
            \begin{align*}
                [F_Y^+,E_X^-] &=\big[[F_Y^+,E_Y^-], E_{r+1}\big]_{q^{-1}}\\
                    &= \frac{1}{q-q^{-1}}[K_Y^{-1} - K_Y,E_{r+1}]_{q^{-1}}\\
                    &= q^{-1}K_Y^{-1}E_{r+1}.
            \end{align*}
        This implies that
            \begin{align*}
                \big[[F_Y^+, \Z_{\t(r)}]_q,F_{r+1}\big]
                    &= -(q-q^{-1})\big[[F_Y^+, E_X^-], F_{r+1} \big]_{q^{-1}} \K{r}\\
                    &= -(1-q^{-2})[ K_Y^{-1}E_{r+1}, F_{r+1}]_{q^{-1}}\K{r}\\
                    &= q^{-1}\K{r}(K_{r+1}^{-1}-K_{r+1})K_Y^{-1}.
            \end{align*}
        As a result we obtain
            \begin{align*}
                \big[B_{r-1}, [[B_r,F_{r+1}]_q, &[F_Y^+,B_{\t(r)}]_q]_q\big]_q - \big[B_{r-1}, [B_r, [F_X^+, B_{\t(r)}]_q ]_q \big]_q\\
                    &= -q(q-q^{-1})^{-1}\cs_{\t(r)}\big[ B_{r-1}, [[F_Y^+, \Z_{\t(r)}]_q,F_{r+1}]]_q\\
                    &= (q-q^{-1})^{-1}\cs_{\t(r)}\big[ B_{r-1}, \K{r}(K_{r+1}- K_{r+1}^{-1})K_Y^{-1} \big]_q\\
                    &= q\cs_{\t(r)}B_{r-1}\K{r}(K_{r+1}^{-1} - K_{r+1})K_Y^{-1}.
            \end{align*}
        By substituting this into \eqref{Eqn:T(r+1)chain2}, we obtain the required result. 
    \end{proof}

\providecommand{\bysame}{\leavevmode\hbox to3em{\hrulefill}\thinspace}
\providecommand{\MR}{\relax\ifhmode\unskip\space\fi MR }
\providecommand{\MRhref}[2]{%
	\href{http://www.ams.org/mathscinet-getitem?mr=#1}{#2}
}
\providecommand{\href}[2]{#2}

\end{document}